\documentclass[letterpaper, 11pt,  reqno]{amsart}

\usepackage[margin=1.2in,marginparwidth=1.5cm, marginparsep=0.5cm]{geometry}

\usepackage{amsmath, mathrsfs, amssymb,amscd,amsthm,amsxtra, esint}
\usepackage{color}

\usepackage[implicit=true]{hyperref}

\usepackage{cases}%%%
\allowdisplaybreaks[2]

\definecolor{gr}{rgb}   {0.,   0.69,   0.23 }
\definecolor{bl}{rgb}   {0.,   0.5,   1. }
\definecolor{mg}{rgb}   {0.85,  0.,    0.85}
\definecolor{yl}{rgb}   {0.8,  0.7,   0.}
\definecolor{or}{rgb}  {0.7,0.2,0.2}

\definecolor{olive}{rgb}{0.3, 0.4, .1}
\definecolor{fore}{RGB}{249,242,215}
\definecolor{back}{RGB}{51,51,51}
\definecolor{title}{RGB}{255,0,90}
\definecolor{dgreen}{rgb}{0.,0.6,0.}
\definecolor{gold}{rgb}{1.,0.84,0.}
\definecolor{JungleGreen}{cmyk}{0.99,0,0.52,0}
\definecolor{BlueGreen}{cmyk}{0.85,0,0.33,0}
\definecolor{RawSienna}{cmyk}{0,0.72,1,0.45}
\definecolor{Magenta}{cmyk}{0,1,0,0}
\definecolor{dorange}{RGB}{154,118,0}
\definecolor{corange}{RGB}{230, 176, 0}

\newtheorem{theorem}{Theorem} [section]

\newtheorem{lemma}[theorem]{Lemma}
\newtheorem{proposition}[theorem]{Proposition}
\newtheorem{remark}[theorem]{Remark}

\newtheorem{definition}[theorem]{Definition}

\DeclareMathOperator*{\intt}{\int}

\DeclareMathOperator*{\supp}{supp}
\DeclareMathOperator{\med}{med}

\DeclareMathOperator{\MAX}{MAX}

\DeclareMathOperator{\Id}{{\bf Id}}
\DeclareMathOperator{\sgn}{sgn}

\newcommand{\I}{\mathcal{I}}

\newcommand{\noi}{\noindent}
\newcommand{\Z}{\mathbb{Z}}
\newcommand{\R}{\mathbb{R}}

\newcommand{\T}{\mathbb{T}}
\newcommand{\bul}{\bullet}

\let\P= \undefined
\newcommand{\P}{\mathbf{P}}

\newcommand{\EE}{\mathcal{E}}

\renewcommand{\H}{\mathcal{H}}

\renewcommand{\L}{\mathcal{L}}

\newcommand{\F}{\mathcal{F}}

\newcommand{\low}{\textup{low}}
\newcommand{\high}{\textup{res}}

\newcommand{\al}{\alpha}

\newcommand{\dl}{\delta}

\newcommand{\eps}{\varepsilon}

\newcommand{\g}{\gamma}

\newcommand{\s}{\sigma}

\newcommand{\ft}{\widehat}
\newcommand{\Ft}{{\mathcal{F}}}
\newcommand{\wt}{\widetilde}

\newcommand{\dx}{\partial_x}
\newcommand{\dt}{\partial_t}
\newcommand{\dd}{\partial}

\newcommand{\ta}{\theta}

\newcommand{\Gdl}{\mathcal{G}_{\dl} }

\newcommand{\Gd}{\wt{\mathcal{G}}_\dl}

\newcommand{\les}{\lesssim}
\newcommand{\ges}{\gtrsim}

\newcommand{\jb}[1]
{\langle #1 \rangle}

\newcommand{\ind}{\mathbf 1}

\renewcommand{\S}{\mathcal{S}}

\newcommand{\N}{\mathbb{N}}

\newcommand{\NN}{\mathcal{N}}

\newcommand{\uu}{\mathbf{u}}
\newcommand{\vv}{\mathbf{v}}

\newcommand{\BO}{\text{\rm BO} }
\newcommand{\KDV}{\text{\rm KdV} }

\newcommand{\too}{\longrightarrow}

\newtheorem*{ackno}{Acknowledgements}

\numberwithin{equation}{section}
\numberwithin{theorem}{section}

\makeatletter
\@namedef{subjclassname@2020}{\textup{2020} Mathematics Subject Classification}
\makeatother

\begin{document}
\baselineskip = 14pt

\title[ILW on the real line]
%{
%A remark on the solution map for the intermediate long wave equation 
%on the real line}
{On the singular nature of 
shallow-water convergence of the intermediate long wave equation \\on the real line}

\author[A.~Chapouto, 
B.~Harrop-Griffiths,
G.~Li, and T.~Oh]
{Andreia Chapouto, 
Benjamin Harrop-Griffiths,
Guopeng Li,  
and Tadahiro Oh}

\address{
Andreia Chapouto\\
CNRS, Laboratoire de math\'ematiques de Versailles, UVSQ, Universit\'e Paris-Saclay, CNRS, 45 avenue des 
\'Etats-Unis, 78035 Versailles Cedex, France, 
and School of Mathematics, Monash University, VIC 3800, Australia}

\email{andreia.chapouto@monash.edu}

\address{Benjamin Harrop-Griffiths, Department of Mathematics and Statistics, Georgetown
University, DC 20057, USA}

\email{benjamin.harropgriffiths@georgetown.edu}

\address{Guopeng Li, 
School of Mathematics and Statistics, Beijing Institute of Technology, Beijing 100081, China}

\email{guopeng.li@bit.edu.cn}

%
%%
%\address{
%Tadahiro Oh, 
%School of Mathematics and Statistics, Beijing Institute of Technology, Beijing 100081, China\\
%and
%School of Mathematics\\
%The University of Edinburgh\\
%and The Maxwell Institute for the Mathematical Sciences\\
%James Clerk Maxwell Building\\
%The King's Buildings\\
%Peter Guthrie Tait Road\\
%Edinburgh\\
%EH9 3FD\\
% United Kingdom}

%
\address{
Tadahiro Oh, 
School of Mathematics\\
The University of Edinburgh\\
and The Maxwell Institute for the Mathematical Sciences\\
James Clerk Maxwell Building\\
The King's Buildings\\
Peter Guthrie Tait Road\\
Edinburgh\\
EH9 3FD\\
 United Kingdom,
 and
 School of Mathematics and Statistics, Beijing Institute of Technology, Beijing 100081, China
}

\email{hiro.oh@ed.ac.uk}

\subjclass[2020]{35Q35, 35Q53, 76B55}

\keywords{intermediate long wave equation;  Korteweg-de Vries equation;
Benjamin-Ono equation}

\begin{abstract}
We investigate 
regularity properties of the solution map
for the intermediate long wave equation (ILW) on the real line.
More precisely, 
we study the scaled ILW 
which was shown to converge
to the Korteweg-de Vries equation (KdV)
in $L^2(\R)$
in the shallow-water limit
in a recent work by 
the first, third, and fourth authors with T.\,Zhao
(2025).
By decomposing the dynamics into 
the low frequency part and the residual part, 
we show that, when the depth parameter is sufficiently small,  
the solution map for the low frequency part
is analytic  in $L^2(\R)$, 
while the 
 solution map for the residual part fails to be~$C^2$.
Moreover, we establish
shallow-water convergence  in $L^2(\R)$
of  the low frequency dynamics
 to  KdV.
This explains the mechanism of  the regularity gain  of the solution map
in the shallow-water limit.

\end{abstract}

%\date{\today}
%%
%
\maketitle
%

%\vspace{-5mm}

\tableofcontents

%\vspace{-10mm}

%%%%%%%%%%%%%%%%%%%%%%%%%%%%%%%%%%%%%%
%%%%%%%%%%%%%%%%%%%%%%%%%%%%%%%%%%%%%%
%%%%%%%%%%%%%%%%%%%%%%%%%%%%%%%%%%%%%%
%%%%%%%%%%%%%%%%%%%%%%%%%%%%%%%%%%%%%%

\newpage

\section{Introduction}
\label{SEC:1}

%\subsection{Intermediate long wave equation}
%\label{SUBSEC:1.1}

We consider  the intermediate long wave equation (ILW)
on the real line:
\begin{align}
\dt u - \mathcal{G}_{\dl}\dx^2 u =\dx(u ^2), 
\label{ILW1}
\end{align}

\noi
where
the unknown $u$ 
denotes the amplitude of the internal wave at the interface
 of a stratified fluid of 
finite depth $0 < \dl < \infty$; 
see \cite{joseph, KKD}.
Here, 
the  operator $\Gdl $ in \eqref{ILW1}
is defined by
\begin{align}
\ft{\Gdl  f}(\xi) = 
-  i\big(\coth(\dl \xi )  - \dl^{-1} \xi^{-1}\big)
\ft f(\xi).
\label{GG1}
\end{align}

The ILW equation \eqref{ILW1}
 serves as an ``intermediate'' equation of finite depth $ \dl \in (0, \infty)$,
bridging 
(i)~the Benjamin-Ono equation (BO), modeling fluid of infinite depth  ($\dl = \infty$):
\begin{align}
\dt u -
 \H \partial_x^2 u  =  \dx (u^2), 
\label{BO}
\end{align}

\noi
where $\H$ denotes the Hilbert transform, 
 and  
(ii)~the Korteweg-de Vries equation, modeling shallow water ($\dl = 0$);
see \eqref{KDV} below.
Moreover,  \eqref{ILW1} is a  completely integrable 
 dispersive PDE
 with an infinite number of conservation laws \cite{Kuper, KSA, KAS82,  
 Matsu, 
 S89}
  and Lax pair structures  \cite{GK, HKV}, thus attracting a wide range of attention;
see~\cite{S19, KS21} for an overview of the subject and the references therein.
In recent years, there have been intensive research activities
 from both deterministic and statistical viewpoints;
see  \cite{ABFS, MST, MV, MPV,  
MPS, Li24, LOZ, IS,
CLOP, CFLOP, FLZ, LP, CLO, GL,HKV, CLOZ} for  well-\,/\,ill-posedness results
as well as  results on 
deep-water  ($\dl \to \infty$)
and shallow-water ($\dl \to 0$)
convergence;
see also 
\cite{CGLLO, GLLO1, GLLO2}.

In studying
 the shallow-water limit ($\dl\to0$), 
we need to magnify the fluid amplitude  by a factor $\sim \frac 1\dl$
 to observe any meaningful limiting behavior.
This leads to  the following scaling~\cite{ABFS}:
\begin{equation*}
v(t,x)  = \tfrac3\dl u\big(\tfrac3\dl t, x\big),
%\label{scale1}
\end{equation*}

\noi
which transforms \eqref{ILW1} to  the following scaled ILW equation:
\begin{align}
\dt v   -    \Gd   \dx^2 v= \dx(v^2), 
\qquad \text{where }\ \Gd := \frac3\dl \Gdl.
\label{sILW1}
\end{align}

\noi
Then, under suitable assumptions, 
the scaled ILW~\eqref{sILW1} is known to  converge to the 
following 
Korteweg-de Vries equation (KdV)
in the shallow-water limit ($\dl\to 0$)
\cite{ABFS, Li24, LOZ, CLOZ}:
\begin{align}
\dt v + \dx^3 v = \dx(v^2) . 
\label{KDV}
\end{align}

\noi
In particular, in a recent work \cite{CLOZ}, 
the first, third, and fourth authors with T.\,Zhao
established shallow-water convergence
of the scaled ILW \eqref{sILW1}
to KdV \eqref{KDV} in $L^2$
on both the real line and the circle $\T = (\R/\Z)$.
%for both $\M = \R$ and $\M = \T = (\R/\Z)$.

\medskip

Our main goal in this paper is
to investigate regularity properties
of the solution map for the scaled ILW \eqref{sILW1}
posed on the real line; see Remark~\ref{REM:circle}
for the periodic case.

It has been known for some time that the solution map exhibits a gain of regularity in the shallow water limit.
On the one hand, Molinet, Saut, and Tzvetkov showed in \cite{MST}
that the solution map 
is not $C^2$-differentiable in $H^s(\R)$
for any $s \in \R$.
On the other hand, 
 the scaled ILW \eqref{sILW1}
 converges 
to KdV \eqref{KDV} in the shallow-water limit,
whose solution map is known to be analytic
in $H^s(\R)$ for $s \ge - \frac 34$; see \cite{BO93, KPV96, Kishi, Guo}.
Our main result, Theorem~\ref{THM:1}, explains how this regularity gain occurs.

We recall that the scaled ILW \eqref{sILW1}
behaves like KdV 
for frequencies $|\xi|\les \dl^{-1}$
(see, for example, Lemma \ref{LEM:res}\,(i)). 
Consequently, we expect that 
the low frequency part of the scaled ILW dynamics
enjoys a better regularity property.
With this in mind, given $0 < \dl \le1$,
we  split the scaled ILW dynamics
into a low frequency part
and a residual part.
Namely, we write a solution $v$ to the scaled ILW \eqref{sILW1}
with initial data $\phi \in L^2(\R)$
as
\begin{align}
v = v^\low + v^\high, 
\label{sILW1a}
\end{align}

\noi
where 
the low frequency part $v^\low$ satisfies
\begin{align}
\begin{cases}
\dt v^\low   -     \Gd   \dx^2 v^\low= \P_{ \dl^{-1}} \dx\big((v^\low)^2\big)\\
v^\low|_{t = 0} = \P_{ \dl^{-1}} \phi
\end{cases}
\label{sILW2}
\end{align}

\noi
and
the residual part $v^\high$ satisfies
\begin{align}
\begin{cases}
\dt v^\high   -     \Gd   \dx^2 v^\high= \dx\big((v^\high)^2\big)
+ 2\dx(v^\high v^\low)
+  \P_{ \dl^{-1}}^\perp \dx\big((v^\low)^2\big)\\
v^\high|_{t = 0} = \P_{ \dl^{-1}}^\perp \phi, 
\end{cases}
\label{sILW3}
\end{align}

\noi
respectively.
Here, $\P_{N}$ denotes the Dirichlet projector onto spatial frequencies $\{|\xi|\leq N\}$
 %(i.e., the Fourier multiplier \(\ind_{|\xi|\le N}\)) 
 and $\P_{N}^\perp = \Id - \P_{N}$. 
 When $\dl = 0$, 
we set $\P_{\dl^{-1}}|_{\dl = 0} = \Id$.
 In particular, 
 we have \(v^\low = \P_{\dl^{-1}}v^\low\) and thus 
\begin{align}
\supp \ft {v^\low}(t, \cdot\,) \subset \{|\xi|\le  \dl^{-1}\}
\label{low1}
\end{align}
for any $t \in \R$.
An immediate consequence is that the \(L^2\)-norm of \(v^\low\) is conserved, which yields global well-posedness of \eqref{sILW2} in $H^s(\R)$ for any $s\geq 0$; see Remark~\ref{REM:GWP1}.
In the following, 
we often refer to the low frequency dynamics \eqref{sILW2}, 
including the case $\dl = 0$, 
with the understanding that, 
when $\dl = 0$,  
it corresponds to KdV \eqref{KDV}
with initial data $v|_{t= 0} = \phi$.

We now state our main result.

\begin{theorem}\label{THM:1}
\textup{(i)}
Let $0 \le \dl \le 1$ and  $s \ge 0$.  Then, given 
any $\phi \in H^s(\R)$, 
there exists small $T = T(\|\phi\|_{H^s})> 0$, 
independent of $0 \le \dl \le 1$, 
such that 
the  map, sending $\phi \in H^s(\R)$
to the solution $v^\low  \in C([0, T]; H^s(\R))$
to \eqref{sILW2}
with initial data  $v^\low|_{t = 0} =\P_{ \dl^{-1}} \phi$, 
is analytic at $\phi$ with respect to the $H^s$-regularity.

\medskip

\noi
\textup{(ii)}
Let $0 < \dl \le 1$ and $s \in \R$.  Then, given any $T > 0$, 
the  map, sending $\phi \in H^s(\R)$
to the solution $v^\high  \in C([0, T]; H^s(\R))$
to \eqref{sILW3}, a priori defined for $s \ge 0$, 
is not $C^2$-differentiable at the zero function with respect to the $H^s$-regularity.

\medskip

\noi
\textup{(iii)}
As $\dl \to 0$, 
the solution $v^\low$ to \eqref{sILW2}
with initial data  $v^\low|_{t = 0} =  \P_{ \dl^{-1}} \phi$
converges to 
the solution $v$ to \eqref{KDV}
with initial data  $v|_{t = 0} =  \phi$
in $C(\R; H^s(\R))$, 
where
$C(\R; H^s(\R))$ is endowed with the compact-open topology in time.

\end{theorem}

The instability in 
Theorem \ref{THM:1}\,(ii)
comes from 
the `low $\times$ high $\mapsto$ high' interaction, 
where high frequencies are much larger than $\dl^{-1}$
(tending to $\infty$), and thus  vanishes in the shallow-water limit.
Hence, Theorem \ref{THM:1}
explains the mechanism of 
 the regularity gain of the solution map in the shallow-water limit, 
from  not being $C^2$ to being analytic.
It also explains how this regularity gain 
propagates in frequency,
namely, from the low frequency dynamics 
supported on $\{|\xi|\le \dl^{-1}\}$ 
to KdV.
This illustrates the singular nature
of the 
 shallow-water limit;
 see Remark \ref{REM:sing}
 for a further discussion
 on various aspects of the singular nature
 of the shallow-water limit.

In light of \eqref{low1}
and the following (crude) product estimate:
\[ \big\|\P_{\dl^{-1}}\big((\P_{\dl^{-1}}f)(\P_{\dl^{-1}} g) \big)    \big\|_{H^s}
\les \dl^{-s - \frac 12 }\| \P_{\dl^{-1}}f \|_{H^s}
\| \P_{\dl^{-1}}g \|_{H^s}
\]

\noi
 for $s\ge 0$ and $0 < \dl \le 1$, 
we can easily construct a solution to 
 the low frequency dynamics~\eqref{sILW2} 
via Picard iterates, 
which provides analyticity 
of the map $\phi \mapsto v^\low$
 for each fixed \(0<\dl\leq 1\).
However, such a construction 
of a solution is restricted
to a time interval $[0, T_\dl]$
for some $T_\dl > 0$ which tends to $0$ as $\dl \to 0$.
 The novelty in Theorem~\ref{THM:1}\,(i) is 
{\it uniformity in $\dl$} of  this analytic behavior.
The proof of this fact
follows from 
a standard contraction argument
for proving local well-posedness, 
once we prove a key bilinear estimate
(Lemma \ref{LEM:bilin1}).
The main observation here is that 
 the low frequency dynamics
\eqref{sILW2}
enjoys the same multilinear dispersion as KdV
(see Lemma~\ref{LEM:res}\,(i)).
Then, 
the key  bilinear estimate
(Lemma~\ref{LEM:bilin1})
follows
from 
a slight modification of the argument by Kenig, Ponce, and Vega \cite{KPV96}. 
Here, an additional difficulty
appears due to the more complicated
nature of the resonance function for the scaled ILW
(see \eqref{LL1})
in carrying out a change of variables
(see \eqref{BX2}-\eqref{BX6}).
Theorem \ref{THM:1}\,(ii)
follows from a straightforward adaptation of 
the argument by 
Molinet, Saut, and Tzvetkov~\cite{MST}
by considering a G\^ateaux derivative 
(at the zero function)
of the solution map
for the residual dynamics \eqref{sILW3}.
We present proofs of Theorem \ref{THM:1}\,(i)
and (ii) in 
Subsection~\ref{SUBSEC:low1}
and Section~\ref{SEC:high}, respectively.

In proving Theorem \ref{THM:1}\,(iii), 
we first establish
 the uniform (in small $\dl> 0$)
equicontinuity  of solutions to 
the low frequency dynamics
(Lemma \ref{LEM:equi})
via a semilinear argument (a contraction-type argument)
{\it without} relying on the complete integrability of
the (scaled) ILW.
Compare this with the situation in \cite{CLOZ}, 
where the complete integrability of the equation 
played a crucial role
in establishing (a weaker version of)
the uniform (in small $\dl > 0$) equicontinuity
of solutions
to the scaled ILW \eqref{sILW1}.
Once we prove Lemma \ref{LEM:equi}, 
the proof of Theorem \ref{THM:1}\,(iii)
reduces to showing convergence
of the ``low frequency part''
of the low frequency dynamics \eqref{sILW2};
see 
 Proposition~\ref{PROP:conv1}
 for details.

\begin{remark}\rm
In \cite{KTz2}, 
Koch and Tzvetkov 
proved 
the failure of local uniform continuity
of the solution map 
for the Benjamin-Ono equation
in $H^s(\R)$
for $s \ge 0$.
By viewing the (original) ILW \eqref{ILW1}
as a perturbation of BO \eqref{BO}
(see \cite{IS, CLOP}), 
it may be possible to  adapt the approach in 
 \cite{KTz2}
to upgrade the failure of $C^2$-regularity 
in Theorem \ref{THM:1}\,(ii)
to  the failure of local uniform continuity
for $s \ge 0$.
For $s < - \frac 12$, the scaled ILW \eqref{sILW1}
is ill-posed in $H^s(\R)$ (see~\cite{CFLOP})
due to the discontinuity of the solution map, 
which may be used to strengthen
Theorem \ref{THM:1}\,(ii).
We, however, do not pursue these issues
for conciseness of the presentation.
We point out that Theorem \ref{THM:1}
suffices to contrast 
different regularity properties
of the low frequency dynamics
and the residual dynamics.

\end{remark}

\begin{remark}\label{REM:circle}\rm

The gain of regularity of the solution map 
for the scaled ILW (Theorem \ref{THM:1})
comes from the BO-like instability 
for the residual dynamics \eqref{sILW3}
due to the `low $\times$ high $\mapsto$ high' interaction.
For the BO equation \eqref{BO}
on the circle, 
Molinet~\cite{Moli2} showed that the solution map
is  analytic
in $H^s_\al(\T)$
for any $s \ge 0$ and $\al \in \R$, 
where 
$H^s_\al(\T)$
denotes the subspace of  $H^s(\T)$
consisting of functions with spatial mean $\al$;
see
\cite[Theorem~1.2]{Moli2}.
See also 
\cite[Theorem~1.2\,(ii)]{GKT0}.
Namely, there is no BO instability 
coming from 
the `low $\times$ high $\mapsto$ high' interaction
in the periodic case.
Similarly, 
the solution map
for ILW \eqref{ILW1} (and thus for  the scaled ILW \eqref{sILW1})
on the circle
is  analytic
in $H^s_\al(\T)$
for any $s \ge 0$ and $\al \in \R$,
which follows from the analyticity
of 
the Birkhoff map \cite{GKT0, GKT}
and a contraction argument 
in the Birkhoff coordinates
\cite[Section 3]{GL}
(together with the Galilean transform to reduce to the mean-zero case).

In particular,
the gain of regularity of the solution map 
for the scaled ILW (Theorem \ref{THM:1})
is 
a phenomenon specific to  the real line case.
Lastly, we point out that, 
while the  solution map for BO  on the circle is
analytic in $H^s_\al(\T)$ for any $s \ge 0$ and $\al \in \R$, 
 there is no known well-posedness result  via a contraction argument
 (directly applied to \eqref{BO})
at this point. 
A similar comment applies to ILW on the circle.

\end{remark}

\begin{remark}\label{REM:low4}\rm
It is possible to include further frequency interactions
as part of the nonlinearity in the low frequency dynamics \eqref{sILW2};
see Subsection \ref{SUBSEC:low2}
(in particular, Remark \ref{REM:low3}).
We, however, believe that the decomposition 
of the scaled ILW \eqref{sILW1}
into the system \eqref{sILW2}-\eqref{sILW3}
is a natural one where the low frequency 
dynamics converges to the KdV dynamics in the shallow-water limit
(Theorem \ref{THM:1}\,(iii)).
Moreover, in view of 
the weakly uniform (in small $\dl > 0$) equicontinuity (\cite[Proposition 1.4]{CLOZ}), 
the high frequency part $\P_{ \dl^{-1}}^\perp v$
can be made arbitrarily small by taking $\dl > 0$ sufficiently small, 
and thus adding high frequency contributions
in~\eqref{sILW2} amounts to adding
small perturbations, which do not seem to be of interest.

\end{remark}

\begin{remark}\label{REM:sing}\rm

%Together with the shallow-water convergence results
%\cite{Li24, CLOZ}, 

Theorem~\ref{THM:1}
exhibits a jump in regularity 
of the solution map for the scaled ILW \eqref{sILW1}.
This comes from the singular nature of the shallow-water convergence
(when compared to the deep-water limit).
As observed in \cite{LOZ, CLO},  
this singular nature of the shallow-water limit
also appears in
the convergence of  the conservation laws 
and invariant measures of the scaled ILW
to those of KdV
(exhibiting 
2-to-1 collapses
with regularity jumps and singular supports of measures);
see 
\cite[Remark~1.2]{CLO}.

\end{remark}

\section{Preliminaries}
\label{SEC:2}

\subsection{Notations}
We use $A\les B$ if there exists $C>0$ such that $A \le CB$, and $A\sim B$ if $A \les B$ and $B \les A$. 
We also write $A \ll B$ if $A \le c B$ for some small constant $c>0$.

The ensuing bilinear estimates (Lemmas \ref{LEM:bilin1}
and
\ref{LEM:bilin2})
will involve expressions that depend on three 
(spatial) frequencies, $\xi, \xi_1, \xi_2 \in \R$. 
We denote the absolute values,  $|\xi|$, $|\xi_1|$, and $|\xi_2|$,
arranged in a non-increasing order by 
 $\xi_{\max}$, 
$\xi_{\med}$, and $\xi_{\min}$, respectively. Typically, the frequencies will satisfy the
 condition: \(\xi = \xi_1 + \xi_2\), in which case we  have \(\xi_{\max}\sim \xi_{\med}\).

Given $\dl \ge 0$, 
we denote
by  
\begin{align}
S_\dl(t) = e^{ t \Gd \dx^2}
\label{lin1}
\end{align}

\noi
 the linear propagator for the scaled ILW \eqref{sILW1}
with the understanding that 
$S_0(t) = e^{-t\dx^3}$ 
is 
 the linear propagator for KdV \eqref{KDV}
 when $\dl = 0$;
 see \eqref{res13}.

\subsection{On the resonance functions}
\label{SUBSEC:res}

In this subsection, we recall the basic properties
of the resonance function for the scaled ILW \eqref{sILW1}.

We first recall that the resonance functions
for KdV \eqref{KDV} and BO \eqref{BO} are given by
\begin{align}
\Xi_\KDV (\bar \xi) &= \xi^3-  \xi_1^3 - \xi_2^3,\notag \\
\Xi_\BO (\bar \xi) &=  |\xi| \xi -  | \xi_1 |\xi_1 - |\xi_2| \xi_2,\label{res2}
\end{align}

\noi
respectively, 
where \(\bar\xi = (\xi,\xi_1,\xi_2)\).
Under $\xi = \xi_1+ \xi_2$, a direct computation shows that 
\begin{align}
\Xi_\KDV (\bar \xi) &= 3\xi \xi_1 \xi_2,\label{res1}\\
|\Xi_{\BO}(\bar \xi)| &= 2\xi_{\med}\xi_{\min}\label{res3}.
\end{align}

Given $0 < \dl < \infty$, 
we now consider the resonance function $\Xi_\dl (\bar \xi)$ for  ILW \eqref{ILW1}, given by 
\begin{align}
\Xi_\dl (\bar \xi)
= \Xi_\dl(\xi,\xi_1,\xi_2) =  p_\dl(\xi)  -   p_\dl(\xi_1) -  p_\dl(\xi_2), 
\label{res4}
\end{align}

\noi
where  $p_\dl(\xi)$ denotes the multiplier for $-i \Gdl \dx^2$:
\begin{align}
p_\dl(\xi)=
\xi^2 \coth(\dl \xi) -\frac{\xi}{\dl} .
\label{res5}
\end{align}

\noi
See \eqref{GG1}.
Then, under $\xi = \xi_1 + \xi_2$, we have
\begin{align}
\Xi_\dl (\bar \xi)
= 
 \xi^2 \coth(\dl \xi)
- \xi_1^2 \coth(\dl \xi_1)
- \xi_2^2 \coth(\dl \xi_2).
\label{res6}
\end{align}

\noi
As in the proof of \cite[Lemma 2.3]{CLOP}, we use the identity
\[
\xi^2 \coth (\dl \xi) - |\xi|\xi = \sgn(\xi) \cdot \frac {2\xi^2} {e^{2|\dl \xi|} - 1}
\]
with the fact that  \(x^2\lesssim e^{2x}-1\) for any \(x\geq 0\) and l'H\^opital's rule (for $\xi =0$) to obtain
\begin{align}
\xi^2 \coth (\dl \xi) - |\xi|\xi = O( \dl^{-2})
\label{res7}
\end{align}

%\noi
%By applying 
%\cite[(2.7)]{CLOP}, 
%the fact that  $x^{\al +1} \le C_\al( e^{2x} - 1)$ for any $x \ge  0$
%and $\al \ge 0$, 
%and l'H\^opital's rule (for $\xi =0$)
%we have 
%\begin{align}
%\xi^2 \coth (\dl \xi) - |\xi|\xi = \sgn(\xi) \cdot \frac {2\xi^2} {e^{2|\dl \xi|} - 1}
%= O( \dl^{-2})
%
%\end{align}

\noi
for any $\xi \in \R$, as $\dl \to \infty$.
Hence, from \eqref{res2}, \eqref{res6},  and \eqref{res7}, 
we have
\begin{align}
\Xi_\dl(\bar \xi) =  \Xi_\BO(\bar \xi) + O(\dl^{-2}), 
\label{res8}
\end{align}

\noi
 as $\dl \to \infty$, 
uniformly in $\bar \xi = (\xi, \xi_1, \xi_2) \in \R^3$
such that $\xi = \xi_1 + \xi_2$.

Lastly, given $0 < \dl < \infty$, we consider the resonance function 
$\wt \Xi_\dl (\bar \xi)$  for the scaled ILW~\eqref{sILW1}, given by 
\begin{align}
\begin{split}
\wt \Xi_\dl (\bar \xi)
& = \wt \Xi_\dl(\xi,\xi_1,\xi_2) 
=  \wt p_\dl(\xi)  -  \wt  p_\dl(\xi_1) -  \wt p_\dl(\xi_2), 
\end{split}
\label{res9}
\end{align}

\noi
where 
$\wt p_\dl(\xi)$
is the 
 multiplier  for $-i\Gd \dx^2$:
 \begin{align}
\wt p_\dl(\xi)
&  =   \frac 3 \dl p_\dl(\xi) 
= \xi \L_\dl(\xi),  %=  \xi^3 \big(1- h(\dl, \xi)  \big), 
\label{res10}
\end{align}

\noi
where $p_\dl(\xi)$ is as in \eqref{res5}
and  
$\L_\dl(\xi) $ is the multiplier for  $\Gd\dx$ given by
\begin{align}
\L_\dl(\xi)
=  \tfrac 3\dl \big( \xi \coth(\dl \xi) - \dl^{-1}\big)
= 6 \xi^2 \sum_{k=1}^\infty \frac{1}{k^2 \pi^2 + \dl^2 \xi^2}, 
\label{res11}
\end{align}

\noi
where the second equality follows
from 
the Mittag-Leffler expansion
of 
$\pi z \coth (\pi z)$ (see \cite[(11) on p.\,189]{Ahl});
see the proof of 
\cite[Lemma 2.3]{LOZ}.
Let us summarize the basic properties
of $\L_\dl(\xi)$;
see \cite[Lemma 8.2.1]{ABFS}, 
\cite[(the proof of) Lemma 2.3]{LOZ}, 
and \cite[Lemma 2.3]{CLO}.

\begin{lemma}\label{LEM:multip}
Let $0< \dl< \infty$. Then, the following statements hold.

\smallskip
\noi{\rm(i)} $0 \le \L_\dl(\xi) \le \min( \xi^2, \tfrac3\dl |\xi| )$.

\smallskip
\noi{\rm(ii)} For fixed $\xi \in \R$, $\L_\dl(\xi)$ is decreasing in $\dl$, and 
\begin{align}
\lim_{\dl\to 0} \L_\dl(\xi) = \xi^2, \quad \xi  \in \R.
\label{res12}
\end{align}

\smallskip
\noi{\rm(iii)} We can also write $\L_\dl(\xi)$ as follows
\begin{align}
\notag
\L_\dl(\xi) 
=  \xi^2 \big( 1 - h(\dl, \xi) \big), 
\end{align}
where $h(\dl, \xi)$ is given by 
\begin{align}
\label{hdef}
h(\dl, \xi)
&
= 6 \dl^2 \xi^2 \sum_{k=1}^\infty \frac{1}{k^2 \pi^2 (k^2 \pi^2 + \dl^2 \xi^2)} \in [0,1].
%\\
%\sum_{n\in\Z_*} h^2(\dl, \xi) 
%&= \infty
%, 
%\quad 
%\text{and}
%\quad 
%\lim_{\dl\to 0} h(\dl, \xi) =0. 
\end{align}

\noi
In particular, from 
 \eqref{hdef}, we have 
\begin{align}
h(\dl, \xi)  \le \dl^2 |\xi|^2
\label{L3}
\end{align}

\noi
for any $\xi \in \R$.

\end{lemma}

In particular, from \eqref{res10} and \eqref{res12}, 
we have
\begin{align}
\lim_{\dl \to 0}
\wt p_\dl(\xi) = \xi^3
\label{res13}
\end{align}

\noi
for each {\it fixed} $\xi \in \R$.
We, however, point out that this convergence
is {\it not} uniform in $\xi \in \R$, 
which is  
the essential difficulty of the shallow-water limit problem
(as compared to the deep-water problem
where the corresponding convergence is uniform in $\xi \in \R$).
We also note that 
from \eqref{res9}, \eqref{res10}, and Lemma \ref{LEM:multip}
with \eqref{res1}, we have
\begin{align}
\begin{split}
\wt \Xi_\dl (\bar \xi)
&
=
 \xi^3 ( 1 - h(\dl, \xi)) - \xi_1^3 (1 - h(\dl,\xi_1)) - \xi_2^3 (1 - h(\dl, \xi_2))
\\
&
= \Xi_\KDV(\bar \xi) - 
\Big( \xi^3 h(\dl,\xi) - \xi_1^3 h(\dl, \xi_1) - \xi_2^3 h(\dl, \xi_2) \Big).
\end{split}
\label{Xi1}
\end{align}

A direct computation  with \eqref{res10} and \eqref{res11}
yields
\begin{align}
\wt \Xi_\dl(\bar \xi)
= 6\xi \xi_1 \xi_2
 \sum_{k=1}^{\infty}   
 \frac{ \pi^2 k^2 \big(3 \pi^2 k^2 + \dl^2(\xi_1^2 + \xi_1 \xi_2 + \xi_2^2)\big)}
 {\prod_{j = 0}^2(\pi^2 k^2 +\dl^2\xi_j^2)}
\label{LL1}
\end{align}

\noi
under  $\xi = \xi_1 + \xi_2$, where $\xi_0 := \xi$.
We recall the following bounds on 
$\wt \Xi_\dl(\bar \xi)$ for $0 < \dl \le 1$;
see \cite[Subsection~5.4]{CGLLO}.

\begin{lemma}\label{LEM:res}
The following holds 
for any $\xi, \xi_1, \xi_2 \in \R$
with $\xi = \xi_1 + \xi_2$, 
uniformly in $0 < \dl \le 1$\textup{:}

\smallskip

\begin{itemize}
\item
[\textup{(i)}]
Suppose that  $\xi_{\max} \les \dl^{-1}$.
Then, we have 
$|\wt \Xi_\dl(\bar \xi)|
\sim |\Xi_\KDV(\bar \xi)|\sim |\xi\xi_1\xi_2|$.

\medskip

\item [\textup{(ii)}]
Suppose that $ \xi_{\max} \gg \dl^{-1}$.
Then, we have 
$|\wt \Xi_\dl(\bar \xi)|
\sim    \dl^{-1} \xi_{\min}\xi_{\max}$.

%\medskip
%
%\noi
%\item[\textup{(iii)}]
%Suppose that 
%$\xi_{\max} \gg    \dl^{-1} \ges  \xi_{\min}$.
%Then, we have 
%$|\wt \Xi_\dl(\bar \xi)|
% \sim    \dl^{-1} \xi_{\min}\xi_{\max}$.

\end{itemize}

\end{lemma}

\begin{proof}
(i) 
Suppose that  $\xi_{\max} \les \dl^{-1}$.
Then,  from \eqref{LL1} and \eqref{res1}, we have
\begin{align*}
|\wt \Xi_\dl(\bar \xi)|
\sim  |\xi \xi_1 \xi_2|
 \sum_{k=1}^{\infty}   
 \frac1{\pi^2 k^2}
\sim  |\xi \xi_1 \xi_2|
\sim |\Xi_\KDV(\bar \xi)|, 
\end{align*}

\noi
uniformly in $0 < \dl \le 1$.

\medskip

\noi
(ii)
We first consider the case  $ \xi_{\min} \gg \dl^{-1}$.
Then,  from \eqref{res3}, 
 we have
\begin{align}
|\Xi_\BO(\bar \xi)|\sim \xi_{\min}\xi_{\max} \gg \dl^{-2}.
\label{LL3}
\end{align}

\noi
Then, 
from 
  \eqref{res8} and  \eqref{LL3}, 
we obtain
\begin{align*}
|\Xi_\dl(\bar \xi) - \Xi_\BO(\bar \xi)| = O(\dl^{-2}) \ll 
 |\Xi_\BO(\bar \xi)|, 
\end{align*}

\noi
from which we conclude that 
\begin{align}
|\Xi_\dl(\bar \xi)|\sim | \Xi_\BO(\bar \xi)|
\label{LL4}
\end{align}

\noi
uniformly in $0 < \dl \le 1$.
Thus, from 
\eqref{LL4} with 
\eqref{res9},  \eqref{res10}, 
and  \eqref{res4},
we obtain
\begin{align*}
|\wt \Xi_\dl(\bar \xi)|\sim \dl^{-1} | \Xi_\BO(\bar \xi)|
\sim    \dl^{-1} \xi_{\min}\xi_{\max}, 
\end{align*}

\noi
uniformly in $0 < \dl \le 1$.

Next, we consider the case
  $\xi_{\max} \gg    \dl^{-1} \ges  \xi_{\min}$.
Without loss of generality, 
assume that $|\xi_1|\le |\xi_2|$.
Since $\xi_{\med} \sim \xi_{\max}$ under $\xi = \xi_1 + \xi_2$, 
it follows from 
 \eqref{LL1} with 
 $\xi_1^2 + \xi_1 \xi_2 + \xi_2^2 \ge
\frac 12  \xi_1^2 + \frac 12 \xi_2^2
\sim \xi_2^2$ that 
\begin{align*}
|\wt \Xi_\dl(\bar \xi)|
\sim |\xi \xi_1 \xi_2|
 \sum_{k=1}^{\infty}   
 \frac{ \pi^2 k^2 }
 {\prod_{j = 0}^1(\pi^2k^2  +\dl^2\xi_j^2)}.
\end{align*}

\noi
Note that the expression on the right-hand side is symmetric in $\xi$ and $\xi_1$.
Without loss of generality, assume that $|\xi| = |\xi_0| \le |\xi_1|$.
Then, we have $|\xi_1| \sim \xi_{\max}
\gg \dl^{-1} \ges \xi_{\min} = |\xi|$. 
In particular, we have 
$\pi^2k^2  +\dl^2\xi^2 \sim \pi^2 k^2$.
Then, by  estimating
the sum  by an integral via a Riemann sum approximation, we have 
\begin{align*}
|\wt \Xi_\dl(\bar \xi)|
& \sim \dl^{-1} |\xi  \xi_2|
 \sum_{k=1}^{\infty}   
 \frac{ 1}
 {\pi^2(\frac k{\dl \xi_1})^2  +1} \frac{1}{\dl |\xi_1|}\\
 & \sim 
    \dl^{-1} \xi_{\min}\xi_{\max}
\int_{0}^\infty
 \frac{ 1}
 {\pi^2x^2 +1}dx\\
& \sim    \dl^{-1} \xi_{\min}\xi_{\max}, 
\end{align*}

\noi
uniformly in $0 < \dl \le 1$.
\end{proof}

\section{Low frequency dynamics}
\label{SEC:low}

\subsection{Proof of Theorem \ref{THM:1}\,(i)}
\label{SUBSEC:low1}

By writing the low frequency dynamics \eqref{sILW2} in the Duhamel formulation, we have
\begin{align}
v^\low(t) =  \P_{ \dl^{-1}} S_\dl(t) \phi
+ \I_\dl\big(  \P_{ \dl^{-1}} (v^\low)^2 \big)(t), 
\label{sILW4}
\end{align}

\noi
where 
 $S_\dl(t) = e^{ t \Gd \dx^2}$
 as in \eqref{lin1}
and
$\I_\dl$ denotes the Duhamel integral operator 
with a derivative, 
given by 
\begin{align}
\I_\dl(F)(t) = \int_0^t S_\dl(t - t') \dx F(t') dt'.
\label{sILW4a}
\end{align}

Given $0 \le \dl < \infty$, 
define the $X^{s, b}_\dl$-space via the norm:
\begin{align}
\|v \|_{X^{s, b}_\dl} = 
\|\jb{\xi}^s \jb{\tau - \wt p_\dl(\xi)}^b \ft v(\tau, \xi)\|_{L^2_{\tau, \xi}}, 
\label{Xsb1}
\end{align}

\noi
where $\wt p_\dl(\xi)$ is as in \eqref{res10}, 
with the understanding that, when $\dl = 0$,  
$\wt p_0(\xi) = \xi^3$ 
in view of~\eqref{res13}.
Given an interval $J \subset\R$
we define the local-in-time version of the $X^{s, b}_\dl$-space
by setting
 \begin{align}
\| v\|_{X^{s,b}_{\dl}(J) }=\inf\big\{\| w\|_{X^{s,b}_{\dl}}: w|_{J}=v\big\}, 
\label{Xsb1b}
\end{align}
     
\noi
where the infimum is taken over all extensions $w$ of $v$
from $J$ to $\R$.
With a slight abuse of notation, 
when $J = [0, T]$, we set 
$X^{s,b}_{\dl}(T) = X^{s,b}_{\dl}([0, T])$.
For $b > \frac 12$, we have 
\begin{align}
\|v\|_{C_JH^s_x}
\les 
\|v\|_{X^{s, b}_{\dl}(J)}, 
% X^{s, b}_{\dl}(J)\subset C(J; H^s(\R)).
 \label{Xsb1a}
\end{align}

\noi
uniformly in 
intervals $J \subset \R$, 
where
$C_{J}H^s_x
= C(J;H^s(\R))$.

Our goal in this subsection is to prove the following uniform  (in small $\dl > 0$) bilinear estimate.

\begin{lemma}
[low $\times$ low $\mapsto$ low]
\label{LEM:bilin1}
Let  $s \ge 0$.
Then, for  $b > \frac 12$ and $b' \le \frac 58 $, we have 
\begin{align}
\big\| \P_{\dl^{-1}}\dx\big( (\P_{ \dl^{-1}}v_1) (\P_{\dl^{-1}}v_2)\big)\big\|_{X^{s, b'-1}_\dl} 
& \les \prod_{j = 1}^2 \| \P_{\dl^{-1}}v_j\|_{X^{s, b}_\dl}, 
\label{bilin1}\\
\big\| \P_{\dl^{-1}}\dx\big( (\P_{ \dl^{-1}}v)^2 \big)\big\|_{X^{s, b'-1}_\dl} 
& \les \| \P_{\dl^{-1}}v\|_{X^{0, b}_\dl}
\| \P_{\dl^{-1}}v\|_{X^{s, b}_\dl} , 
\label{bilin1x}
\end{align}

\noi
uniformly in $0 \le \dl \le1$.

\end{lemma}

%The first estimate \eqref{bilin1} in  

Lemma \ref{LEM:bilin1}
 together with the standard linear estimates (see, for example,  \cite{TAO}), 
 which hold uniformly in $0 \le  \dl \le 1$, 
yields
local well-posedness of 
\eqref{sILW2} 
in $X^{s,b}_{\dl}(\tau)$ with $b = \frac 12 + \g$
for sufficiently small $\g > 0$, 
via  a contraction argument applied to the Duhamel formulation~\eqref{sILW4}.
We note that a standard 
persistence-of-regularity argument
with the second estimate~\eqref{bilin1x},
which holds uniformly in $0 \le \dl \le 1$,  
allows us to  
choose 
the local existence time
 $\tau  = \tau(\|\phi\|_{L^2})>0$ 
to depend only on the $L^2$-norm of the initial data $\phi$, 
 independent of $0 \le  \dl \le 1$ and $s \ge 0$, 
while guaranteeing that 
\begin{align}\label{bilin1y}
\|v^\low\|_{X^{s,b}_\dl(\tau)}\lesssim \|\phi\|_{H^s},
\end{align}

\noi
uniformly in 
$0 \le  \dl \le 1$ (and $s \ge 0$). 
Furthermore, we point out that, in view of \eqref{Xsb1a}
and the quadratic (in particular, analytic) nonlinearity in \eqref{sILW4}, 
we conclude that 
the map, sending $\phi \in H^s(\R)$
to the solution $v^\low  \in C([0, \tau]; H^s(\R))$
to \eqref{sILW2}, 
is analytic at $\phi$ with respect to the $H^s$-regularity.
This proves Theorem \ref{THM:1}\,(i).

\begin{remark}\label{REM:GWP1}
\rm 

In view of \eqref{low1}, by writing \eqref{sILW2}
as 
\begin{align}
\dt v^\low   -     \Gd   \dx^2 v^\low= \P_{ \dl^{-1}} \dx\big(( \P_{ \dl^{-1}}v^\low)^2\big), 
\label{sILW2a}
\end{align}

\noi
it is easy to see that the $L^2$-norm of a solution to \eqref{sILW2a}
is conserved.
Hence, 
in view of the persistence of regularity
mentioned above
(namely, 
 the local existence time depends only on the $L^2$-norm of the data), 
we conclude  that  the low frequency dynamics \eqref{sILW2}
is globally well-posed in $H^s(\R)$ for $s \ge 0$
(without referring
 to the global well-posedness of the scaled ILW \eqref{sILW1}
in $L^2(\R)$ established in \cite{IS, CLOP}).
As for global well-posedness in $L^2(\R)$
of KdV (corresponding to $\dl = 0$), see \cite{BO93, KPV96}.

\end{remark}

Before proceeding to  a proof of 
Lemma \ref{LEM:bilin1}, 
we 
 recall the $L^4$-Strichartz estimate \cite{BO93}:
\begin{align}
\| v \|_{L^4_{t, x}} \les 
\|v\|_{X^{0, \frac 38}_\dl}, 
\label{Str1}
\end{align}

\noi
uniformly in $0 < \dl \le 1$, 
which follows from a straightforward modification
of the periodic version (\cite[Lemma 2.11]{Li24}).

\begin{proof}[Proof of Lemma \ref{LEM:bilin1}]
We only prove the first bound \eqref{bilin1},
since the second bound \eqref{bilin1x}
follows from 
the first bound \eqref{bilin1}
with $s = 0$
and the triangle inequality: $\jb{\xi}^s \les \jb{\xi_1}^s + \jb{\xi_2}^s$
for $\xi = \xi_1 + \xi_2$ when $s \ge 0$.

When $\dl = 0$ (corresponding to the KdV case), 
the bilinear estimate~\eqref{bilin1} 
was proven in~\cite{KPV96}.
Thus, we restrict our attention to $0 < \dl \le 1$ in the following.

We set 
\begin{align}
\s = \jb{\tau - \wt p_\dl(\xi)}
\qquad \text{and}\qquad 
\s_j = \jb{\tau_j - \wt p_\dl(\xi_j)}, \quad j = 1, 2, 
\label{mod1}
\end{align}

\noi
where $\wt p_\dl(\xi)$ is as in \eqref{res10}.
Then, 
it suffices to prove
\begin{align}
\bigg\|
\intt_{\substack{\tau = \tau_1 + \tau_2\\\xi = \xi_1 + \xi_2}}
\frac{\ind_{  |\xi|\le  \dl^{-1}} \cdot \xi}{\s^{1-b'}}
\frac{ f_1(\tau_1, \xi_1)f_2(\tau_2, \xi_2) d\tau_1 d\xi_1}
{\s_1^{b}
\s_2^{b}}
\bigg\|_{L^2_{\tau, \xi}}
\les \prod_{j = 1}^2 \|f_j\|_{L^2_{\tau, \xi}}, 
\label{mod2}
\end{align}

\noi
uniformly in $0 < \dl \le1$, 
for any functions $f_1, f_2 \in L^2(\R^2)$
such that 
$\supp f_j(\tau_j, \cdot\,) \subset \{|\xi_j|\le   \dl^{-1}\}$
for any $\tau_j \in \R$, $j = 1, 2$.
%where $\s$ and $\s_j$, $j = 1, 2$, 
%are as in \eqref{mod1}.

\medskip

\noi
$\bul$ {\bf Case 1:}
$|\xi - 2\xi_1|\ges |\xi|$.
\\
\indent
Proceeding with the 
 Cauchy-Schwarz argument as in \cite{KPV96}, 
 we see that~\eqref{mod2} follows once we prove 
 the following bound: 
 \begin{align}
\begin{split}
 \sup_{\tau, \xi \in \R} \frac{\ind_{  |\xi|\le   \dl^{-1}}\cdot |\xi|}{\s^{1-b'}}
 \bigg(\intt_{\substack{\tau = \tau_1 + \tau_2\\\xi= \xi_1 + \xi_2}}
\frac{
\ind_{|\xi - 2\xi_1|\ges |\xi|}\cdot 
\prod_{j = 1}^2 \ind_{ |\xi_j|\le   \dl^{-1}} d\tau_1 d\xi_1}
{\s_1^{2b}
\s_2^{2b}}
\bigg)^\frac 12 
< \infty, 
\end{split}
\label{BD1}
\end{align}

\noi
uniformly in $0 < \dl \le1$.
For simplicity of notation, 
we drop the frequency restrictions in the following
but it is understood that we work under the condition: 
\begin{align}
%1\ll |\xi| \le   \dl^{-1}, 
%\qquad 
|\xi|,  |\xi_1|, |\xi_2|\le   \dl^{-1}
 \qquad\text{and}
\qquad 
 |\xi - 2\xi_1|\ges |\xi|.
\label{BD1a}
\end{align}

%In view of the symmetry, we assume $|\xi_1|\le |\xi_2|$ in the following.
From \cite[Lemma 2.3]{KPV96}
with $\xi = \xi_1 + \xi_2$ and $2b > 1$, we have 
\begin{align}
\intt_{\tau = \tau_1 + \tau_2}
\frac{d\tau_1 }{\s_1^{2b}
\s_2^{2b}}
\les 
\frac{1}{\jb{\tau  - \wt p_\dl(\xi) + \wt \Xi_\dl(\bar \xi)}^{2b}}, 
\label{BD2}
\end{align}

\noi
where $\s_j$, $j = 1, 2$, is as in \eqref{mod1} and $\wt \Xi_\dl(\bar \xi)$ is as in \eqref{res9}.
For integration in $\xi_1$, 
we change variables:
\begin{align}
\mu = \mu_{\dl, \tau, \xi}(\xi_1) = \tau  - \wt p_\dl(\xi) + \wt \Xi_\dl(\xi, \xi_1, \xi - \xi_1).
\label{BD3}
\end{align}

\noi
From
\eqref{LL1}, we have
\begin{align}
\begin{split}
 \dd_{\xi_1} & \wt \Xi_\dl(\xi, \xi_1,  \xi_2) \\
& = 3\xi  \xi_2
 \sum_{k=1}^{\infty}   
 \frac{2 \pi^2 k^2 \big(3 \pi^2 k^2 + \dl^2(\xi_1^2 + \xi_1 \xi_2 + \xi_2^2)\big)}
 {\prod_{j = 0}^2(\pi^2 k^2 +\dl^2\xi_j^2)}\\
& \quad +  3\dl \xi \xi_1 \xi_2
 \sum_{k=1}^{\infty}   
 \frac{2 \pi^2 k^2 }
 {\prod_{j = 0}^2(\pi^2 k^2 +\dl^2\xi_j^2)}\\
& \hphantom{XXXXXX} \times
 \bigg\{ \dl (2\xi_1 + \xi_2) 
- \frac{  2\dl \xi_1  \big(3 \pi^2 k^2 + \dl^2(\xi_1^2 + \xi_1 \xi_2 + \xi_2^2)\big)}
 {\pi^2 k^2 +\dl^2\xi_1^2}\bigg\}. 
\end{split}
\label{BX1}
\end{align}

\noi
Then, from \eqref{BD3} and \eqref{BX1}
with $\xi_2 = \xi - \xi_1$
and the symmetry 
of $ \wt \Xi_\dl(\xi, \xi_1,  \xi_2)$ in $\xi_1$ and $\xi_2$, 
 we have 
\begin{align}
\begin{split}
 \frac{d \mu}{d \xi_1}
& = (\dd_{\xi_1} - \dd_{\xi_2})
\wt \Xi_\dl(\xi, \xi_1,  \xi_2) \\
& = 3\xi  (\xi -2 \xi_1)
 \sum_{k=1}^{\infty}   
 \frac{2 \pi^2 k^2 \big(3 \pi^2 k^2 + \dl^2(\xi_1^2 + \xi_1 \xi_2 + \xi_2^2)\big)}
 {\prod_{j = 0}^2(\pi^2 k^2 +\dl^2\xi_j^2)}\\
& \quad +  \underbrace{3\dl \xi \xi_1 \xi_2
 \sum_{k=1}^{\infty}   
 \frac{2 \pi^2 k^2 }
 {\prod_{j = 0}^2(\pi^2 k^2 +\dl^2\xi_j^2)}A_{\dl, k}(\bar \xi)}_{=:M_{\dl}(\bar \xi) },  \\
%& = : \1 + \II, 
\end{split}
\label{BX2}
\end{align}

\noi
where $A_{\dl, k}(\bar \xi) = A_{\dl, k}(\xi, \xi_1, \xi_2)$ is given by 
\begin{align*}
A_{\dl, k}(\bar \xi)
& = -\dl(\xi_2 - \xi_1) \\
& \quad + \big(3 \pi^2 k^2 + \dl^2(\xi_1^2 + \xi_1 \xi_2 + \xi_2^2)\big)
\bigg\{
\frac{  2\dl \xi_2  }
 {\pi^2 k^2 +\dl^2\xi_2^2}
- \frac{  2\dl \xi_1  }
 {\pi^2 k^2 +\dl^2\xi_1^2}\bigg\}.
\end{align*}

\noi
A direct computation yields
\begin{align}
A_{\dl, k}(\bar \xi)
& = \frac{\dl(\xi_2 - \xi_1) }
 {\prod_{j = 1}^2 (\pi^2 k^2 +\dl^2\xi_j^2)}
 B_{\dl, k}(\bar \xi) 
\label{BX4}
\end{align}

\noi
with  $B_{\dl, k}(\bar \xi) = B_{\dl, k}(\xi, \xi_1, \xi_2)$  given by 
\begin{align*}
 B_{\dl, k}(\bar \xi)
&  = 
5\pi^4 k^4 + \pi^2 k^2 \dl^2 (\xi_1^2 - 4\xi_1\xi_2 + \xi_2^2)
 - \dl^4 \xi_1\xi_2(2\xi_1^2 + 3\xi_1\xi_2+ 2\xi_2^2).
\end{align*}

\noi
Using the frequency restriction \eqref{BD1a}:
$ |\xi|, 
 |\xi_1|, |\xi_2|\le   \dl^{-1}$, 
 we then have
\begin{align}
\frac{ |B_{\dl, k}(\bar \xi)| }
 {\prod_{j = 1}^2 (\pi^2 k^2 +\dl^2\xi_j^2)}
\le 5 + \frac 6{\pi^2} + \frac 7{\pi^4} 
\approx 5.67 < 6, 
\label{BX5}
\end{align}

\noi
uniformly in $0 < \dl \le 1$ and $k \in \N$.
Thus, from \eqref{BX2}, \eqref{BX4}, and
\eqref{BX5}  with \eqref{BD1a}, 
we obtain
\begin{align}
M_{\dl}(\bar \xi) & = 3\xi  (\xi -2 \xi_1)
 \sum_{k=1}^{\infty}   
 \frac{2 \pi^2 k^2 }
 {\prod_{j = 0}^2(\pi^2 k^2 +\dl^2\xi_j^2)}
  D_{\dl, k}(\bar \xi)
\label{BX6}
\end{align}

\noi
for  some $D_{\dl, k}(\bar \xi) =  D_{\dl, k}(\xi, \xi_1, \xi_2)$
with $| D_{\dl, k}(\bar \xi)| <  6$, 
uniformly in 
 $0 < \dl \le 1$, $k \in \N$, 
and  $ |\xi|, 
 |\xi_1|, |\xi_2|\le   \dl^{-1}$.
Hence, by noting that 
\[
3\pi^2 k^2 \le 
3 \pi^2 k^2 + \dl^2(\xi_1^2 + \xi_1 \xi_2 + \xi_2^2)
\le 3\pi^2 k^2  + 3, 
\]

\noi
it follows from \eqref{BX2} and \eqref{BX6} that 
 the sign of $\dd_{\xi_1} \mu$
agrees with the sign of $\xi(\xi - 2\xi_1)$
(which in particular implies that,  for fixed $\xi$, we can carry out the change of variables~\eqref{BD3}
by splitting the domain of integration into two regions)
and that 
\begin{align}
& |\dd_{\xi_1} \mu |
\sim |\xi (\xi - 2\xi_1)| \ges \xi^2, 
\label{BD4}
\end{align}

\noi
uniformly in $0 < \dl \le 1$, 
 where the last bound follows from the last assumption in~\eqref{BD1a}.
Hence, from \eqref{BD2}, \eqref{BD3},  and \eqref{BD4}, we have
\begin{align*}
\text{LHS of \eqref{BD1}}
& \les \sup_{\tau, \xi \in \R} \frac{|\xi|}{\s^{1-b'}}
\bigg( \int \frac{d\xi_1}{\jb{\tau  - \wt p_\dl(\xi) + \wt \Xi_\dl(\xi, \xi_1, \xi- \xi_1)}^{2b}}\bigg)^\frac{1}{2}\\
& \les \sup_{\tau, \xi \in \R} \frac{1}{\s^{1-b'}}
\bigg( \int \frac{d\mu}{\jb{\mu}^{2b}}\bigg)^\frac{1}{2}\\
& < \infty, 
\end{align*}

\noi
provided that $b > \frac 12$ and $b' \le 1$.
This proves \eqref{BD1}.

\medskip

\noi
$\bul$ {\bf Case 2:}
$|\xi - 2\xi_1|\ll |\xi|$.
\\
\indent
In this case, we have 
\begin{align}
 |\xi|\sim |\xi_1| \sim |\xi_2| \le  \dl^{-1} .
\label{bilin1aa}
\end{align}
Then, from Lemma \ref{LEM:res}\,(i), 
we have 
\begin{align}
\MAX := \max(\s, \s_1, \s_2) \ges | \xi|^3.
\label{bilin1a}
\end{align}

We first consider the case $|\xi |\les 1$.
In this case, 
we do not need to use the lower bound~\eqref{bilin1a} 
on the modulations.
By H\"older's inequality and  the $L^4$-Strichartz estimate \eqref{Str1}, we have
\begin{align*}
\text{LHS of \eqref{bilin1}}
\les \prod_{j = 1}^2 \| \P_{ \dl^{-1}} v_j\|_{L^4_{t, x}}
\les \prod_{j = 1}^2 \| \P_{ \dl^{-1}} v_j\|_{X^{s, b}_\dl}, 
%\label{mod0}
\end{align*}

\noi
uniformly in $0 < \dl \le 1$, provided that
$b \ge \frac 38$
and 
 $b' \le 1$.

In the following, we assume $1 \ll |\xi|\le  \dl^{-1}$.
Suppose that 
$\MAX = \s$.
From \eqref{bilin1a} with~\eqref{bilin1aa}, we have 
\begin{align}
\frac{ |\xi|}{\s^{1-b'}}
\les |\xi|^{-2 + 3b'}
\les 1,
\label{bilin1c} 
\end{align}

\noi
provided that $b' \le \frac 23 $.
Then, \eqref{mod2} follows from \eqref{bilin1c}, 
H\"older's inequality (on the physical side), and the $L^4$-Strichartz estimate~\eqref{Str1}, 
provided that $b \ge \frac 38$.

Next, we consider the case
$\MAX = \s_j$, $j = 1$ or $2$.
From \eqref{bilin1a} with \eqref{bilin1aa}, we have 
\begin{align}
\frac{ |\xi|}{\s_j^{b}}
\les |\xi|^{1-3b}
\les 1,
\label{bilin1d} 
\end{align}

\noi
provided that $b \ge \frac 13  $.
Then, the bound \eqref{mod2}
(and hence \eqref{bilin1})
follows from 
duality, polarization,  
\eqref{bilin1d}, 
H\"older's inequality, and the $L^4$-Strichartz estimate~\eqref{Str1}, 
which requires $1 - b', b \ge \frac 38$.
\end{proof}

\subsection{Shallow-water convergence of the low frequency dynamics}
\label{SUBSEC:A}

In this subsection,  we establish 
shallow-water convergence in $H^s(\R)$, $s\ge 0$, 
of the low frequency dynamics~\eqref{sILW2}
to the KdV dynamics \eqref{KDV} (Theorem \ref{THM:1}\,(iii)).
We first recall that
both the low frequency dynamics~\eqref{sILW2}
and KdV \eqref{KDV}
are globally well-posed in $H^s(\R)$, $s \ge 0$; 
see
 Remark \ref{REM:GWP1}.

 Theorem~\ref{THM:1}\,(iii)
 follows once we prove the following proposition.

\begin{proposition}\label{PROP:conv1}
Let $s \ge 0$.
 Given $\phi \in H^s(\R)$, 
let $\{\phi_\dl\}_{0 < \dl \le 1} \subset H^s(\R)$
such that $\phi_\dl$ converges to $\phi$ in $H^s(\R)$ as $\dl \to 0$.
Let $v$ and $v_\dl^\low$
be the global-in-time solutions to KdV~\eqref{KDV} with $v|_{t = 0} = \phi$ and 
to the low frequency dynamics
\eqref{sILW2}
 with $v_\dl^\low|_{t = 0} = \P_{ \dl^{-1}} \phi_\dl$, 
 respectively.
Then, $v_\dl^\low$
converges to $v$ in $C(\R; H^s(\R))$
as $\dl \to 0$, 
where
$C(\R; H^s(\R))$ is endowed with the compact-open topology in time.

\end{proposition}

In the following, 
 we only consider positive times.

\medskip

In proving Proposition \ref{PROP:conv1}, 
uniform (in $0 < \dl \le 1$) equicontinuity 
of solutions to the low frequency dynamics \eqref{sILW2} plays a key role.
For this purpose, 
 let us recall the notion of equicontinuity.

\begin{definition}\label{DEF:EC}
\rm 
Let $s \in \R$.
We say that a bounded set $E \subset H^s(\R)$ is equicontinuous
if 
\begin{align}
 \lim_{\eps \to 0+} \sup_{f \in E} \sup_{|y|< \eps} 
\| f (\,\cdot + y) - f(\,\cdot\,)\|_{H^s} = 0.
\label{EEC1}
\end{align}

\noi
By Plancherel's identity, \eqref{EEC1} is equivalent to 
\begin{align*}
 \lim_{N\to \infty} \sup_{f \in E}
 \| \P_{N}^\perp f\|_{H^s} = 0.
\end{align*}

\end{definition}

The following lemma establishes
uniform (in $0 < \dl \le 1$) equicontinuity 
in $H^s(\R)$, $s \ge 0$, 
of solutions to the low frequency dynamics \eqref{sILW2}
as well
as
local-in-time uniform (in $0 < \dl \le 1$) ``equicontinuity''
with respect to the $X^{s, b}$-norm.

\begin{lemma}\label{LEM:equi}
Given $s \ge 0$, 
let $\{\phi_\dl\}_{0 \leq  \dl \le 1} \subset H^s(\R)$
be bounded and equicontinuous.
Let  $v_\dl^\low$
be the global-in-time solution
to  the low frequency dynamics
\eqref{sILW2}
 with initial data $v_\dl^\low|_{t = 0} = \P_{ \dl^{-1}} \phi_\dl$.\footnote{Recall our convention
 that, when $\dl = 0$,  
 $v_0^\low$ corresponds to the solution to KdV \eqref{KDV}
 with initial data $v_0^\low|_{t = 0} = \phi_0$.}
 Then, 
given any $T > 0$,
we have 
\begin{align}
\lim_{N \to \infty}
 \sup_{0 \leq  \dl \le 1}
 \| \P_{N}^\perp v_\dl^\low\|_{X_\dl^{s, \frac 12 + \g}(T)} = 0
\label{EC0}
\end{align}
 
\noi
for sufficiently small $\g > 0$,
and consequently,
\begin{align}
 \lim_{N\to \infty} \sup_{0 \le  \dl \le 1}\sup_{0 \le t \le T}
 \| \P_{N}^\perp v_\dl^\low(t)\|_{H^s} = 0.
\label{EC1}
\end{align}

\end{lemma}

We note that 
Lemma \ref{LEM:equi} follows from 
the semilinear local well-posedness of \eqref{sILW2}
and
the translation characterization \eqref{EEC1} of equicontinuity 
for the 
$X^{s,b}(T)$-norm; see \eqref{Ben3}.
Compare this with the  weakly uniform equicontinuity
in $L^2(\R)$
of solutions to the scaled ILW \eqref{sILW1}
proven in \cite{CLOZ}
which relied heavily on the complete integrability of
the equation.

\begin{proof}[Proof of Lemma \ref{LEM:equi}]
In the following, various constants
depend on $s\ge 0$ and small $\g > 0$
but we suppose their dependence on $s$ and $\g$
for simplicity of the presentation.

Fix $T > 0$.
By the persistence of regularity
(namely, iteratively applying \eqref{bilin1x} in Lemma~\ref{LEM:bilin1}
with 
 the $L^2$-conservation
 and 
\eqref{Xsb1a})
and
 the boundedness of   $\{\phi_\dl\}_{0 \leq  \dl \le 1} \subset H^s(\R)$, 
  we have 
\begin{align}
\begin{split}
\sup_{0\leq \dl\leq 1}\sup_{y\in \R}\|v^\low _\dl(\,\cdot + y) \|_{C_T H^s_x}
& \le C(T) \sup_{0\leq \dl\leq 1}\sup_{y\in \R}\|\phi_\dl(\,\cdot + y) \|_{H^s}\\
& \leq C(T) \sup_{0\leq \dl\leq 1}\|\phi_\dl\|_{H^s} < \infty, 
\end{split}
\label{Bex1}
\end{align}

\noi
where
$C_{T}H^s_x
= C([0, T];H^s(\R))$
and
the translation by $y$ on the left-hand side is in the spatial variable.
Then, 
it follows from a slight modification of the local well-posedness argument
with~\eqref{Bex1} (see also \eqref{bilin1y})
that 
there exists  \(\tau = \tau\big(T,\sup_{0\leq \dl\leq 1}\|\phi_\dl\|_{H^s}\big)>0\) 
such that 
\begin{align}
\|v_\dl^\low(\, \cdot + y) - v_\dl^\low \|_{X^{s,\frac12+\gamma}_\dl([t_0,t_0+\tau])}\lesssim \|v_\dl^\low(t_0,\, \cdot + y) - v_\dl^\low(t_0)\|_{H^s}
\label{Bex2}
\end{align}

\noi
for any $t_0 \in [0, T]$ and $y \in \R$.
Together with \eqref{Xsb1a}, 
this  
 in turn implies that 
 there exists a constant \(K = K\big(T,\sup_{0\leq \dl\leq 1}\|\phi_\dl\|_{H^s}\big)>0\) such that
\begin{align}
\|v_\dl^\low(\, \cdot + y) - v_\dl^\low \|_{C_TH^s_x}\leq K\|\phi_\dl(\,\cdot+y) - \phi_\dl\|_{H^s}, 
\label{Bex3}
\end{align}

\noi
uniformly in $0 \le \dl  \le 1$ and 
 $y \in \R$.

Let
$I_j  = \big[\frac j 2 \tau, (\frac j 2 + 1)\tau\big]$, 
$j = 0, 1, \dots, 2 \big\lceil\frac T\tau \big\rceil$, 
where $\lceil x\rceil$ denotes the smallest integer
greater than or equal to $x$.
Then, we have $|I_j| = \tau$,  $|I_j\cap I_{j+1}| = \frac \tau 2$, 
and 
\[ [0, T] \subset \bigcup_{j = 0}^{\lceil T/\tau \rceil}I_j.\]

\noi
Hence, 
by  a gluing argument (see, for example,  \cite[Lemma 4.5]{bring})
with \eqref{Bex2} and \eqref{Bex3}, 
we obtain, for some 
 \(C = C\bigl(T,\sup_{0\leq \dl\leq 1}\|\phi_\dl\|_{H^s}\bigr)>0\), 
\begin{align*}
\|v_\dl^\low(\, \cdot + y) - v_\dl^\low \|_{X^{s,\frac12+\gamma}_\dl(T)}\leq C\|\phi_\dl(\,\cdot+y) - \phi_\dl\|_{H^s}, 
\end{align*}

\noi
uniformly in $0 \le \dl  \le 1$ and 
 $y \in \R$.
Consequently, 
from the equicontinuity of $\{\phi_\dl\}_{0 \le \dl \le 1}$, we obtain
\begin{align}
\lim_{\eps\to0+}\sup_{0\leq \dl\leq 1}\sup_{|y|<\eps}
\|v_\dl^\low(\cdot + y) - v_\dl^\low\|_{X^{s,\frac12+\gamma}_\dl(T)} = 0.
\label{Ben2}
\end{align}

Now, we note that 
\begin{align}
 \| \P_{N}^\perp v_\dl^\low\|_{X_\dl^{s, \frac 12 + \g}(T)}^2
 \lesssim \int_\R \|v_\dl^\low(\,\cdot+y) - v_\dl^\low\|_{X_\dl^{s,\frac12+\g}(T)}^2 Ne^{-2N|y|}\,dy, 
\label{Ben3}
\end{align}

\noi
which follows from 
\[ 1 - \ind_{[-N, N]}(\xi) \les \
\int_\R  |e^{i \xi y} - 1|^2Ne^{-2N|y|}dy.\]

\noi
See
\cite[the proof of Lemma 4.2]{KV19}.
See also 
\cite[Theorem 4]{Pego}
for a related argument.
We also note that by iterating the local-in-time argument  for \eqref{sILW2}
using \eqref{bilin1x} in Lemma \ref{LEM:bilin1}
together with the $L^2$-conservation (see Remark \ref{REM:GWP1}), 
we have 
\begin{align}
 \sup_{0 \leq  \dl \le 1}
 \|v_\dl^\low\|_{X_\dl^{s,\frac12+\g}(T)}
& \le C\Big(T, 
 \sup_{0\leq \dl\leq 1}\|\phi_\dl\|_{H^s}\Big).
 \label{Ben3a}
\end{align}

Fix  $\eta > 0$.
Then, it follows from \eqref{Ben2} that 
there exists
 $\eps = \eps(\eta) > 0$ such that 
\begin{align}
\sup_{0\leq \dl\leq 1}\sup_{|y|<\eps}
\|v_\dl^\low(\cdot + y) - v_\dl^\low\|_{X^{s,\frac12+\gamma}_\dl(T)}^2 < \frac 12 \eta.
\label{Ben4}
\end{align}

\noi
Hence, by noting that $\int_{|y|< \eps} N e^{-2N|y|} \le 1$, 
it follows from \eqref{Ben3} with \eqref{Ben4} and \eqref{Ben3a} that 
there exists
 $N_* = N_*\big(\eta, \eps,  T, \sup_{0\leq \dl\leq 1}\|\phi_\dl\|_{H^s}\big) > 0$ 
 such that 
\begin{align*}
 \sup_{0 \leq  \dl \le 1}
 \| \P_{N}^\perp v_\dl^\low\|_{X_\dl^{s, \frac 12 + \g}(T)} ^2
& < 
\frac 12 \eta
+ 
4
 \sup_{0 \leq  \dl \le 1}
 \|v_\dl^\low\|_{X_\dl^{s,\frac12+\g}(T)}^2
\int_{|y|\ge \eps} Ne^{-2N|y|}\,dy\\
& < \eta
\end{align*}

\noi
for any $N \ge N_*$.
This proves 
 \eqref{EC0}.
The second limit  \eqref{EC1}  follows from the fact that 
$X^{s,\frac12+\gamma}_\dl(T)\subset C([0,T];H^s(\R))$.
\end{proof}

We now present a proof of Proposition \ref{PROP:conv1}.

\begin{proof}[Proof of Proposition \ref{PROP:conv1}]
Without loss of generality, we assume that 
$\phi_ \dl = \P_{ \dl^{-1}} \phi_\dl$
such that we do not need to carry the frequency projector
$\P_{ \dl^{-1}}$ on the initial data.
As before,  we drop the superscript ``$\low$''
and set 
$v_\dl = v_\dl^\low$.
We also set $\phi_0 = \phi$
and use $v_0 = v_{\dl = 0}$ to denote $v$.

Let $\vv$ and $\vv_\dl$
denote the interaction representations
of $v$ and $v_\dl$,  defined by 
\begin{align}
\vv(t) = S_0(-t) v(t)\qquad \text{and}
\qquad 
\vv_\dl(t) = S_\dl(-t) v_\dl(t), 
\label{int1}
\end{align}

\noi
respectively.
In the following, we may use $\vv_0 = \vv_{\dl = 0}$ to denote $\vv$.
Recall that we have
\begin{align}
\begin{split}
\| v\|_{X^{s, b}_0} 
& =
\| \vv\|_{H^{s, b}} 
:= 
\|\jb{\dx}^s \jb{\dt}^b \vv\|_{L^2_{t, x}} , 
\\
\| v_\dl\|_{X^{s, b}_\dl} 
& =
\| \vv_\dl\|_{H^{s, b}} , 
\end{split}
\label{int2}
\end{align}

\noi
where 
 the $X^{s, b}_\dl$-norm is as in 
\eqref{Xsb1}.
As in~\eqref{Xsb1b},
given an interval $J \subset\R_+$, 
 we define the 
local-in-time version of the $H^{s, b}$-space
by setting
 \begin{align*}
\| u\|_{H^{s, b}(J)} =\inf\big\{\| w\|_{H^{s, b}}: w|_{J}=u\big\}.
\end{align*}

\noi
With a slight abuse of notation, 
when $J = [0, T]$, we set 
$H^{s, b}(T) = H^{s, b}([0, T])$.

Let $ s\ge 0$ and fix a target time $T \ge 1$.
Our goal is to prove
\begin{align}
\lim_{\dl\to 0}\|v - v_\dl\|_{C_{T}H^s_x}
= 0.
\label{BBen1}
\end{align}

\noi
By 
Bernstein's inequality, we have 
\begin{align}
\|v - v_\dl\|_{C_{T}H^s_x}
\les\sup_{0\le \dl \le 1} \|\P_N^\perp v_\dl\|_{C_{T}H^s_x} 
+ N^s\|\P_N(v - v_\dl)\|_{C_{T}L^2_x}, 
\label{BBen2}
\end{align}

\noi
uniformly in $0 < \dl \le 1$ and $N > 0$.
Fix small $\eps > 0$.
From Lemma \ref{LEM:equi}, 
there exists $N_1 = N_1(\eps) > 0$
such that 
\begin{align}
 \sup_{0 \le  \dl \le 1}\sup_{0 \le t \le T}
 \| \P_{N}^\perp v_\dl(t)\|_{ H^s} < \eps
\label{BBen3}
\end{align}

\noi
for any $N \ge N_1$.
Thus, \eqref{BBen1} follows 
from \eqref{BBen2} and \eqref{BBen3}, once we prove 
\begin{align}
\lim_{\dl \to 0}\|\P_N(v - v_\dl)\|_{C_{T}L^2_x}= 0
\label{BBen4}
\end{align}

\noi
for each {\it fixed} sufficiently large $N\gg1$.

Hence, we focus on proving \eqref{BBen4} in the following.
Fix $N \gg1$.
From \eqref{int1}, 
we have 
\begin{align}
\begin{split}
\| \P_N(v - v_\dl ) \|_{C_{T}L^2_x}
& \le 
\sup_{t \in [0, T]}\| \P_N(S_0(t) - S_\dl(t))\vv(t)\|_{L^2_x}\\
& \quad + \|\P_N(\vv - \vv_\dl)\|_{C_{T}L^2_x}.
\end{split}
\label{A0a}
\end{align}

\noi
The  first term on the right-hand side of \eqref{A0a}
can be treated as in the proof of \cite[Theorem~1.2]{CLOZ}.
For readers' convenience, we present some details.
Given small $\eps > 0$, it follows from the uniform continuity of $\vv$ on the
compact time interval $[0, T]$
with 
 the unitarity in $L^2(\R)$ of $S_0(t)$ and $S_\dl(t)$
 that there exist $M \in \N$ 
 and $\{t_j\}_{j = 1}^M \subset [0, T]$ such that 
\begin{align}
\begin{split}
&\sup_{t \in [0, T]}\| \P_N (S_0(t) - S_\dl(t) ) \vv(t) \|_{L^2} \\
& \quad
<  \max_{j = 1, \dots, M}\| (S_0(t_j) - S_\dl(t_j) )  \vv(t_j) \|_{L^2} 
+ \frac \eps 2, 
\end{split}
\label{A0b}
\end{align}

\noi
uniformly in $N \gg1 $.
From  the mean value theorem
and \eqref{res13}, we have
\begin{align*}
\big| \Ft_x\big( (S_0(t) -  S_\dl(t) ) f \big) (\xi) \big|
= 
|t|  |\xi^3 - \wt p_\dl(\xi) |
|\ft f(\xi)|
\too 0 
\end{align*}

\noi
as $\dl\to0$
for each fixed $t, \xi \in \R$.
Thus, it follows from 
the dominated convergence theorem
that there exists small $\dl_1  = \dl_1(\eps, T, \phi)> 0$ such that 
\begin{align}
\max_{j = 1, \dots, M}\| (S_0(t_j) - S_\dl(t_j) )  \vv(t_j) \|_{L^2} 
< \frac \eps 2
\label{A0c}
\end{align}

\noi
for any $0 < \dl < \dl_1$.
Due to the use of the dominated convergence theorem, 
$\dl_1$ depends on the profile of the initial data $\phi$
(and not just on its $L^2$-norm).
Hence, from \eqref{A0b} and \eqref{A0c}, 
we conclude that 
\begin{align}
\lim_{\dl \to 0}
\sup_{t \in [0, T]}
\| \P_N (S_0(t) - S_\dl(t))\vv(t)\|_{L^2} = 0, 
\label{A0d}
\end{align}

\noi
uniformly in $N \gg1 $.
Therefore, 
in view of \eqref{A0a} and \eqref{A0d}, 
the desired limit \eqref{BBen4} (and thus Proposition~\ref{PROP:conv1}) follows once we prove 
\begin{align}
\lim_{\dl \to 0}\|\P_N(\vv - \vv_\dl)\|_{C_{T}L^2_x}= 0
\label{XX1}
\end{align}

\noi
for each fixed sufficiently large $N\gg1$.

In the remaining part of the proof, 
we fix
 $\tau \in (0, 1]$
to be  the uniform (in $0 \le \dl \le 1$)  local existence time
in $L^2(\R)$
for a solution $v_\dl$ to \eqref{sILW2}
(and~\eqref{KDV} when $\dl = 0$), 
depicted by Lemma~\ref{LEM:bilin1}
and the boundedness of $\{\phi_\dl\}_{0 \le \dl \le 1}$ in $L^2(\R)$.
Then, in view of the $L^2$-conservation
with~\eqref{bilin1y}
and~\eqref{int2}, 
we have 
\begin{align}
\sup_{0 \le \dl\le 1} \|\vv_\dl\|_{H^{0, \frac 12 + \g}([t_0, t_0 + \tau])} 
\les \sup_{0 \le \dl\le 1} \|\phi_\dl\|_{L^2}\les 1, 
\label{IC0c}
\end{align}

\noi
uniformly in  $t_0 \ge 0$, 
where
 $\g > 0$ is sufficiently small as in Lemma~\ref{LEM:equi}.

Fix small $\eps > 0$.
Then, given small $\eps' > 0$
(to be chosen later in terms of $\eps > 0$), 
it follows from 
Lemma~\ref{LEM:equi} with \eqref{int2} that 
there exists $N_2 = N_2( \eps')> 0$ 
such that 
\begin{align}
 \sup_{0 \le  \dl \le 1}
 \| \P_{N}^\perp \vv_\dl\|_{H^{0, \frac 12 + \g}([t_0, t_0 + \tau])} 
= 
 \sup_{0 \le  \dl \le 1}
 \| \P_{N}^\perp v_\dl\|_{X_\dl^{0, \frac 12 + \g}([t_0, t_0 + \tau])} \le \eps'
\label{IC0b}
\end{align}

\noi
for any $N \ge N_2$ and $t_0 \ge 0$

In the following, we study the contribution 
from the frequencies $\{|\xi|\le N\}$
for $t_0 = 0$.
From \eqref{KDV} and \eqref{sILW2}, we have 
\begin{align}
\begin{split}
\P_N (\vv - \vv_\dl)
& = \P_N (\phi - \phi_\dl)
+ \P_N \NN_0(\vv)
- 
\P_N \NN_\dl(\vv_\dl)\\
& 
= \P_N (\phi - \phi_\dl)
+ \Big(\P_N \NN_\dl(\vv)
- 
\P_N \NN_\dl(\vv_\dl)\Big) 
+  \EE_{N, \dl}(\vv)
\end{split}
\label{ID1}
\end{align}

\noi
for sufficiently small $0 < \dl \ll N^{-1}$, 
where $\NN_0(\uu)$ and $\NN_\dl(\uu)$
are given by 
\begin{align}
\begin{split}
\ft {\NN_0(\uu)}(t, \xi) 
&  =i \xi  
\int_0^t 
\intt_{\xi = \xi_1 + \xi_2}e^{-it' \Xi_\KDV(\bar\xi)}\ft \uu(t', \xi_1)
\ft \uu(t', \xi_2)d\xi_1dt',\\
\ft {\NN_\dl(\uu)}(t, \xi)  
& =
 i \xi  
 \int_0^t 
\intt_{\substack{\xi = \xi_1 + \xi_2\\|\xi_1|, |\xi_2|\le \dl^{-1}}}e^{-it' \wt \Xi_\dl(\bar\xi)}
\ft \uu (t', \xi_1)\ft \uu(t', \xi_2)d\xi_1dt'
\end{split}
\label{ID2}
\end{align}

\noi
with 
$\Xi_\KDV(\bar \xi)$ and 
$\wt \Xi_\dl(\bar\xi)$ as in 
 \eqref{res1}
 and 
\eqref{res9}, respectively, 
and the error term 
$\EE_{N, \dl}(\uu)$ is given by 
\begin{align}
\EE_{N, \dl}(\uu) = \P_N\big( \NN_0(\uu) - \NN_\dl(\uu)\big).
\label{ID3}
\end{align}

Since $\phi_\dl$ converges to $\phi$ in $L^2(\R)$
as $\dl \to 0$, 
there exists small $\dl_2 = \dl_2(\eps) > 0$, 
independent of $N \in \N$,  such that 
\begin{align}
\|\P_N (\phi - \phi_\dl)\|_{H^{0, \frac 12 + \g}(\tau)}
\le C\|\P_N (\phi - \phi_\dl)\|_{L^2}
< \frac \eps 6
\label{ID4}
\end{align}

\noi
for any $0 < \dl < \dl_2$.
From the nonhomogeneous linear estimate, Lemma \ref{LEM:bilin1} 
(both at the level of the interaction representation), 
\eqref{IC0b}, and \eqref{IC0c}, we have, 
for some small $\ta > 0$, 
\begin{align}
\begin{split}
& \|\P_N \NN_\dl(\vv)
- 
\P_N \NN_\dl(\vv_\dl)\|_{H^{0, \frac 12 + \g}(\tau)}\\
& \quad \le CT^\ta 
\Big(\|\P_N (\vv - \vv_\dl) \|_{H^{0, \frac 12 + \g}(\tau)}
+ \|\P_N^\perp (\vv - \vv_\dl) \|_{H^{0, \frac 12 + \g}(\tau)}\Big)\\
& \quad \le C'T^\ta
\Big(\|\P_N (\vv - \vv_\dl) \|_{H^{0, \frac 12 + \g}(\tau)}
+ \eps'\Big)\\
& \quad \le \frac 12 
\|\P_N (\vv - \vv_\dl) \|_{H^{0, \frac 12 + \g}(\tau)}
+ \frac \eps 6, 
\end{split}
\label{ID5}
\end{align}

\noi
where the last inequality follows
from possibly choosing smaller $\tau > 0$
(by a constant factor) and $0 < \eps' = \eps'(\eps) \ll \eps$.

It remains to estimate the error term $\EE_{N, \dl}(\vv)$
in \eqref{ID1} and \eqref{ID3}.
We assume that $\phi\ne 0$ (and hence $\vv \not \equiv 0$)
since $\EE_{N, \dl}(0) = 0$.
By applying Lemma~\ref{LEM:bilin1}
(which also holds for $\dl = 0$)
with \eqref{IC0c}, we have 
\begin{align}
\|\EE_{N, \dl}(\vv)\|_{H^{0, \frac 12 + 2\g}(\tau)}
\les 1, 
\label{ID6}
\end{align}

\noi
uniformly in $0 < \dl \le 1$.
From \eqref{ID3} and \eqref{ID2}, 
we have 
\begin{align}
\EE_{N, \dl}(\vv) = \EE_{N, \dl}^{(1)}(\vv) + \EE_{N, \dl}^{(2)}(\vv), 
\label{ID7}
\end{align}

\noi
where $\EE_{N, \dl}^{(1)}(\vv)$ and $\EE_{N, \dl}^{(2)}(\vv)$ are given by 
\begin{align}
\begin{split}
& \ft {\EE_{N, \dl}^{(1)}(\vv)}(t, \xi) \\
& \quad = \ind_{|\xi|\le N}\cdot
i \xi  
\int_0^t 
\intt_{\xi = \xi_1 + \xi_2}
\Big(e^{-it' \Xi_\KDV(\bar\xi)} - e^{-it' \wt \Xi_\dl(\bar\xi)}\Big)
\ft \vv(t', \xi_1)
\ft \vv(t', \xi_2)d\xi_1dt',\\
& \ft {\EE_{N, \dl}^{(2)}(\vv)}(t, \xi) \\
&\quad = \ind_{|\xi|\le N}\cdot
 i \xi  
\int_0^t 
\hspace{-3mm}
\intt_{\substack{\xi = \xi_1 + \xi_2\\ \max(|\xi_1|, |\xi_2|)> \dl^{-1}}}
e^{-it' \wt \Xi_\dl(\bar\xi)}
\ft \vv (t', \xi_1)\ft \vv(t', \xi_2)d\xi_1dt'.
\end{split}
\label{ID8}
\end{align}

We first treat $\EE_{N, \dl}^{(1)}(\vv)$.
Fix small $\eps'' > 0$.
Then, 
given small $\ta > 0$, 
it follows from Lemma~\ref{LEM:equi}
(which also holds for $\dl = 0$)
that 
\begin{align}
 \|\P_{\dl^{-\frac 25 +\ta}}^\perp \vv\|_{C_{T}L^2_x} \ll N^{-\frac 32}   \eps''
 \|\phi\|_{L^2}^{-1}
\label{ID9}
\end{align}

\noi
for any sufficiently small
$\dl = \dl(\eps'',   N)>0$.

\smallskip

\noi
$\bullet$ {\bf Case 1:}
 $\max\big(|\xi_1|, |\xi_2|\big) \le \dl^{-\frac 25 +\ta}$.
\\
\indent
In this case, we have  $|\xi| \les \dl^{-\frac 25 +\ta}$.
Then, 
from \eqref{Xi1} and \eqref{L3}, we have 
\begin{align}
|\Xi_\KDV (\bar \xi)
-\wt  \Xi_\dl (\bar \xi)|\les \dl^{5\ta}.
\label{ID10}
\end{align}

\noi
Thus, 
by the mean value theorem with \eqref{ID10}, we have 
\begin{align}
\big|e^{-it' \Xi_\KDV(\bar\xi)} - e^{-it' \wt \Xi_\dl(\bar\xi)}\big|
\les T \dl^{5\ta}
\label{ID11}
\end{align}

\noi
for any $0 \le t' \le T$.
Hence, 
from \eqref{ID8}, H\"older's inequality (on the Fourier side), and \eqref{ID11}
with the conservation of the $L^2$-norm, 
we have
\begin{align}
\begin{split}
\|\EE_{N, \dl}^{(1)}(\vv)\|_{C_{\tau_0} L^2_x}
& \le C T \dl^{5\ta} N^\frac 32 \|\vv\|_{C_{T}L^2_x}^2
= C T \dl^{5\ta} N^\frac 32 \|\phi\|_{L^2}^2\\
& <  \eps''
\end{split}
\label{ID12}
\end{align}

\noi
for any sufficiently small
$\dl = \dl(\eps'',   N, T)>0$
and $0 < \tau_0 \le1 $.

\smallskip

\noi
$\bullet$ {\bf Case 2:}
 $\max\big(|\xi_1|, |\xi_2|\big) > \dl^{-\frac 25 +\ta}$.
\\
\indent
Without loss of generality, assume 
$|\xi_1| >  \dl^{-\frac 25 +\ta}$.
Then, 
from \eqref{ID8} and  H\"older's inequality (on the Fourier side)
with \eqref{ID9} and the conservation of the $L^2$-norm, 
we have
\begin{align}
\begin{split}
\|\EE_{N, \dl}^{(1)}(\vv)\|_{C_{\tau_0} L^2_x}
& \le C    N^\frac 32 \| \P^\perp_{\dl^{-\frac 25 + \ta}} \vv\|_{C_{T}L^2_x}
 \|\phi\|_{L^2}\\
& < \eps''
\end{split}
\label{ID13}
\end{align}

\noi
for any $0 < \tau_0 \le1 $.

\medskip

As for $\EE_{N, \dl}^{(2)}(\vv)$, 
we can proceed as in Case 2 above and obtain
\begin{align}
\begin{split}
\|\EE_{N, \dl}^{(2)}(\vv)\|_{C_{\tau_0} L^2_x}
 <  \eps''
\end{split}
\label{ID14}
\end{align}

\noi
for any 
sufficiently small
$\dl = \dl(\eps'',   N)>0$
and 
$0 < \tau_0  \le1 $.
Putting \eqref{ID7}, \eqref{ID12}, \eqref{ID13}, and \eqref{ID14}
together, we conclude that 
\begin{align}
\begin{split}
\lim_{\dl \to 0 } \|\EE_{N, \dl}(\vv)\|_{H^{0, 0}(\tau)} 
& = \lim_{\dl \to 0 } \|\EE_{N, \dl}(\vv)\|_{L^2_\tau L^2_x} 
\le \tau^\frac 12 \lim_{\dl \to 0 } \|\EE_{N, \dl}(\vv)\|_{C_\tau L^2_x} \\
& = 0. 
\end{split}
\label{ID15}
\end{align}

\noi
Finally, 
given small $\eps > 0$, 
it follows from 
interpolating~\eqref{ID6} and \eqref{ID15}
that there exists $\dl_3 = \dl_3(\eps, N, T) > 0$
such that 
\begin{align}
\|\EE_{N, \dl}(\vv)\|_{H^{0, \frac 12 + \g}(\tau)}
< \frac \eps 6
\label{ID16}
\end{align}

\noi
for any $0 < \dl < \dl_3$.

Hence, from \eqref{ID1}, \eqref{ID4}, \eqref{ID5}, and \eqref{ID16}, 
we obtain
\begin{align*}
\|\P_N (\vv - \vv_\dl) \|_{H^{0, \frac 12 + \g}(\tau)}
< \eps
\end{align*}

\noi
for any $0 < \dl < \min(\dl_2, \dl_3)$, 
from which we obtain
\begin{align}
\lim_{\dl \to 0 }
\|\P_N (\vv - \vv_\dl) \|_{C_\tau L^2_x}
\les 
\lim_{\dl \to 0 }
\|\P_N (\vv - \vv_\dl) \|_{H^{0, \frac 12 + \g}(\tau) }=0.
\label{ID17}
\end{align}

\noi
By iterating the  argument, 
we have
\begin{align}
\lim_{\dl \to 0 }
\|\P_N (\vv - \vv_\dl) \|_{H^{0, \frac 12 + \g}([j \tau, (j+1)\tau] \cap [0, T]) }=0
\label{ID18}
\end{align}

\noi
for $j = 1, \dots, \big[\frac {T}\tau \big]$, 
where $[x]$ denotes the integer part of $x \in \R$.
Therefore, from \eqref{ID17} and~\eqref{ID18}, we obtain
\begin{align*}
& \lim_{\dl \to 0 }
 \|\P_N (\vv - \vv_\dl) \|_{C_{T} L^2_x}\\
& \quad \les 
\lim_{\dl \to 0 }
\sum_{j = 0}^{[T/\tau]}
\|\P_N (\vv - \vv_\dl) \|_{H^{0, \frac 12 + \g}([j \tau, (j+1)\tau]) \cap [0, T]) }=0
\end{align*}

\noi
for each $N \ge N_2$, where $N_2 = N_2(\eps') = N_2(\eps)$ is as in 
\eqref{IC0b}.
This proves \eqref{XX1}
and thus 
 concludes the proof of Proposition \ref{PROP:conv1}.
\end{proof}

\subsection{On other frequency interactions}
\label{SUBSEC:low2}

In Subsection \ref{SUBSEC:low1},  we treated the case $\xi_{\max} \le \dl^{-1}$.
In this subsection, we consider several other frequency interactions.

\begin{lemma}[high $\times$ high $\mapsto$ low] \label{LEM:bilin2}
Let $s \ge 0$.
Then, for  $b > \frac 12$ and $b' \le \frac 58 $, we have 
\begin{align}
\big\| \P_{ \dl^{-1}}\dx\big( (\P_{ \dl^{-1}}^\perp v_1) (\P_{ \dl^{-1}}^\perp v_2)\big)\big\|_{X^{s, b'-1}_\dl} 
\les \prod_{j = 1}^2 \| \P_{ \dl^{-1}}^\perp v_j\|_{X^{s, b}_\dl}, 
\label{bilin2}\\
\big\| \P_{ \dl^{-1}}\dx\big( (\P_{ \dl^{-1}}^\perp v)^2\big)\big\|_{X^{s, b'-1}_\dl} 
\les \| \P_{ \dl^{-1}}^\perp v\|_{X^{0, b}_\dl}\| \P_{ \dl^{-1}}^\perp v\|_{X^{s, b}_\dl}, 
\label{bilin2x}
\end{align}

\noi
uniformly in $0 < \dl \le1$.

\end{lemma}

\begin{proof} %[Proof of Lemma \ref{LEM:bilin2}]
In the following, we only prove \eqref{bilin2}
since \eqref{bilin2x} follows from a straightforward modification.
As in the proof of Lemma \ref{LEM:bilin1}, 
let $\xi$ denote the spatial frequency 
of 
$ \P_{ \dl^{-1}}\dx\big( (\P_{ \dl^{-1}}^\perp v_1) (\P_{ \dl^{-1}}^\perp v_2)\big)$, 
appearing on the left-hand side of \eqref{bilin2}.
We first consider the case $|\xi |\les 1$.
In this case, by H\"older's inequality and  the $L^4$-Strichartz estimate \eqref{Str1}, we have
\begin{align*}
\text{LHS of \eqref{bilin2}}
\les \prod_{j = 1}^2 \| \P_{ \dl^{-1}}^\perp v_j\|_{L^4_{t, x}}
\les \prod_{j = 1}^2 \| \P_{ \dl^{-1}}^\perp v_j\|_{X^{s, b}_\dl}, 
%\label{BDDx}
\end{align*}

\noi
uniformly in $0 < \dl \le 1$, provided that
$b \ge \frac 38$
and 
 $b' \le 1$.

Next, we consider the case $|\xi|\gg 1$.
It suffices to prove %\eqref{BD0}
\begin{align}
\bigg\|
\intt_{\substack{\tau = \tau_1 + \tau_2\\\xi= \xi_1 + \xi_2}}
\frac{\ind_{ 1\ll |\xi|\le  \dl^{-1}} \cdot \xi}{\s^{1-b'}}
\frac{ f_1(\tau_1, \xi_1)f_2(\tau_2, \xi_2) d\tau_1 d\xi_1}
{\s_1^{b}
\s_2^{b}}
\bigg\|_{L^2_{\tau, \xi}}
\les \prod_{j = 1}^2 \|f_j\|_{L^2_{\tau, \xi}}, 
\label{bilin2a}
\end{align}

\noi
uniformly in $0 < \dl \le1$, 
for any functions $f_1, f_2 \in L^2(\R^2)$
such that 
$\supp f_j(\tau_j, \cdot\,) \subset \{|\xi_j|>  \dl^{-1}\}$
for any $\tau_j \in \R$, $j = 1, 2$, 
where $\s$ and $\s_j$, $j = 1, 2$, 
are as in \eqref{mod1}.
For simplicity of notation, 
we drop the frequency restrictions in the following
but it is understood that we work under the condition: 
\begin{align}
1\ll |\xi|\le  \dl^{-1} < |\xi_1|, |\xi_2|.
\label{bilin2aa}
\end{align}

In view of the symmetry, we assume $|\xi_1|\le |\xi_2|$ in the following.
This in particular implies $|\xi_2|\ges |\xi |\gg 1$.
From Lemma \ref{LEM:res}\,(ii) with~\eqref{bilin2aa}, 
we have 
\begin{align}
\MAX  \sim \dl^{-1} |\xi \xi_2|, 
\label{bilin2b}
\end{align}

\noi
where $\MAX$ is as in \eqref{bilin1a}.

\medskip

\noi
$\bul$ {\bf Case 1:}
$\MAX = \s$.
\\
\indent
From \eqref{bilin2b} with \eqref{bilin2aa}, we have 
\begin{align}
\frac{ |\xi|}{\s^{1-b'}}
\les \frac{|\xi|^{b'}}{\dl^{-1 + b'} |\xi_2|^{1-b'}}
\les \frac{|\xi|^{2b'-1}}{ |\xi_2|^{1-b'}}
\les 1,
\label{bilin2c} 
\end{align}

\noi
provided that $b' \le \frac 23 $.
Then, \eqref{bilin2a} 
(and hence \eqref{bilin2})
follows from \eqref{bilin2c}, 
H\"older's inequality (on the physical side), and the $L^4$-Strichartz estimate~\eqref{Str1}, 
provided that $b \ge \frac 38$.

\medskip

\noi
$\bul$ {\bf Case 2:}
$\MAX = \s_j$, $j = 1$ or $2$.
\\
\indent
From \eqref{bilin2b} with \eqref{bilin2aa}, we have 
\begin{align}
\frac{ |\xi|}{\s_j^{b}}
\les \frac{|\xi|^{1-b}}{\dl^{-b} |\xi_2|^{b}}
\les |\xi|^{1-3b}
\les 1,
\label{bilin2d} 
\end{align}

\noi
provided that $b \ge \frac 13  $.
Then, the bound \eqref{bilin2a}
(and hence \eqref{bilin2})
follows from 
duality, polarization,  
\eqref{bilin2d}, 
H\"older's inequality, and the $L^4$-Strichartz estimate~\eqref{Str1}, 
which requires $1 - b', b \ge \frac 38$.
\end{proof}

\begin{remark}\label{REM:low3}\rm
(i)  In view of Lemma \ref{LEM:bilin2}, 
it is also possible to write $v = v_1 + v_2$, 
where 
$v_1$ and $v_2$ satisfy
\begin{align}
\begin{cases}
\dt v_1  -     \Gd   \dx^2 v_1= \P_{ \dl^{-1}} \dx(v_1^2)\\
v_1|_{t = 0} =  \phi
\end{cases}
\label{sILWx1}
\end{align}

\noi
and
\begin{align*}
\begin{cases}
\dt v_2   -     \Gd   \dx^2 v_2= \dx(v_2^2)
+ 2\dx(v_2 v_1)
+  \P_{ \dl^{-1}}^\perp \dx(v_1^2)\\
v_2|_{t = 0} = 0, 
\end{cases}
%\label{sILWx2}
\end{align*}

\noi
respectively.
Here, we did not include the frequency projector
$\P_{ \dl^{-1}}$ on the initial data in~\eqref{sILWx1}, 
since, if we did, then~\eqref{sILWx1} would reduce
to \eqref{sILW2}.
Then, an analogue of Theorem~\ref{THM:1}
with $v^\low$ and $v^\high$ replaced
by $v_1$ and $v_2$, respectively, holds.
Similarly, by a slight modification
of the proof of Lemma \ref{LEM:bilin2}
(with $b' > b > \frac12$, both sufficiently close to $\frac 12$), 
we may include 
the nonlinearity coming from 
$ \dl^{-1} < |\xi| \les |\xi_1|^{1-\ta} \sim |\xi_2|^{1-\ta}$
(for some small $\ta > 0$)
as part of the $v_1$-equation.

\medskip

\noi
(ii)
We may also consider the interaction coming from 
$ \dl^{-1} \ll |\xi| \les |\xi_1| \sim |\xi_2|$
but $|\xi| \gg |\xi_2|^{1-\ta}$
(for some small $\ta > 0$).
In this case, 
from Lemma \ref{LEM:res}\,(ii), we only have
$|\wt \Xi_\dl(\bar \xi)|
\sim    \dl^{-1} \xi_{\min}\xi_{\max}
\sim \dl^{-1} |\xi\xi_{2}|$,
which forces us to take $b = b' = \frac 12$
if we were to proceed as in the proof of Lemma \ref{LEM:bilin2}.
In fact, by working with the $Z_\dl^{s, b}$-norm given by 
\begin{align*}
\|u \|_{Z^{s, b}_\dl} = 
\|u \|_{X^{s, b}_\dl} 
+ 
\|\jb{\xi}^s \jb{\tau - \wt p_\dl(\xi)}^{b- \frac 12} \ft u(\tau, \xi)\|_{L^2_{\xi}L^1_\tau}, 
%\label{Xsb1x}
\end{align*}

\noi
we can prove a bilinear estimate (with $b = b' = \frac 12$)
to treat the nonlinearity coming from 
$ |\xi| \les |\xi_1| \sim |\xi_2|$;
see the proof of 
\cite[Theorem 3.19]{ETz2}
in the KdV case.
If we proceed this way, however, 
there is an issue in establishing 
an analogue of Lemma \ref{LEM:bilin1}
with the $Z^{s, b}$-norm (with $b = b' = \frac 12$).
First of all, the step \eqref{BD2} breaks down.
Moreover, for $|\xi_1|\ll 1 \les |\xi|\sim |\xi_2| \les \dl^{-1}$, 
it follows from Lemma \ref{LEM:res}\,(i) that 
\begin{align*}
\frac{ |\xi|}{\MAX^\frac 12 }
\les \frac{|\xi|}{|\xi\xi_1\xi_2|^\frac 12 }
\sim \frac{1}{|\xi_1|^\frac 12}, 
\end{align*}

\noi
where the right-hand side can be arbitrarily large
due to the low frequency problem particular
to the real line case.
While it may be possible to reconcile the issues
(to include
both 
$\xi_{\max} \les \dl^{-1}$
and
$ |\xi| \les |\xi_1| \sim |\xi_2|$ as part of the ``low frequency'' dynamics), 
we do not pursue this issue further in this paper.
See also Remark \ref{REM:low4}.
\end{remark}

\section{Failure of $C^2$-regularity of the residual dynamics}
\label{SEC:high}

In this section, we present a proof of Theorem \ref{THM:1}\,(ii).
This type of instability was first established by Bourgain \cite{BO97};
see also \cite{Tzv}.
In the following, we closely 
follow the argument in \cite{MST}
for BO \eqref{BO}, 
exploiting 
the `low $\times$ high $\mapsto$ high' interaction.
We present details for readers' convenience
(and to reflect the difference between our problem and that in \cite{MST}).

Fix $\phi \in \S(\R)$.
Given small $\eps > 0$,
let  $v  = v(\eps)$ be the smooth solution to the scaled ILW~\eqref{sILW1}
with initial data $v|_{t = 0} = \eps \phi$.
Then, by writing  it as
$v = v^\low + v^\high$ as in \eqref{sILW1a}, 
a direct computation with \eqref{sILW4}
(where $\phi$ is replaced by $\eps \phi$)
gives
\begin{align}
\begin{split}
v^\low(\eps)|_{\eps = 0} & = 0, \\
\dd_\eps v^\low(\eps)|_{\eps = 0} 
& =  \P_{ \dl^{-1}} S_\dl(t) \phi.
%\dd_\eps^2 v^\low(\eps)|_{\eps = 0} 
%& = \I_\dl\Big( \P_{ \dl^{-1}} (\P_{ \dl^{-1}} S_\dl(t) \phi)^2\Big).
\end{split}
\label{der1}
\end{align}

\noi
Similarly,
by writing 
the residual dynamics \eqref{sILW3} in the Duhamel formulation:
\begin{align}
\begin{split}
v^\high(t) 
& =  \P_{ \dl^{-1}}^\perp S_\dl(t) \phi
+ \I_\dl\big( (v^\high)^2\big)(t)\\
& \quad 
+ 2 \I_\dl( v^\high v^\low)(t)
+ \P_{ \dl^{-1}}^\perp \I_\dl\big( (v^\low)^2\big)(t), 
\end{split}
\label{sILW5}
\end{align}

\noi
a direct computation with \eqref{sILW5}
(where $\phi$ is replaced by $\eps \phi$)
and \eqref{der1}
yields
\begin{align}
\begin{split}
v^\high(\eps)|_{\eps = 0} & = 0, \\
\dd_\eps v^\high(\eps)|_{\eps = 0} 
& =  \P_{ \dl^{-1}}^\perp S_\dl(t) \phi,\\
\dd_\eps^2 v^\high(\eps)|_{\eps = 0} 
& = 
\I_\dl\Big( (\P_{ \dl^{-1}}^\perp S_\dl(t) \phi)^2\Big)\\
& \quad + 2 \I_\dl\Big( \P_{ \dl^{-1}} S_\dl(t) \phi \cdot \P_{ \dl^{-1}}^\perp S_\dl(t) \phi\Big)\\
& \quad + \I_\dl\Big( \P_{ \dl^{-1}}^\perp \big((\P_{ \dl^{-1}} S_\dl(t) \phi)^2\big)\Big).
\end{split}
\label{der2}
\end{align}

Fix \(t\in \R\backslash\{0\}\), $s \in \R$, and $0 < \dl \le 1$.
Given a positive integer $N \gg \dl^{-1} \ge 1$ and small  $\al = \al(\dl, N)> 0$ such that 
\begin{align}
\dl^{-1} \al N \sim N^{-\ta}
\label{ill1}
\end{align}

\noi
 for some small $\ta > 0$, 
 define $\phi = \phi(N, \al)$ by its Fourier transform:
 \begin{align}
\ft \phi (\xi) = \al^{-\frac 12 }
\Big\{\ind_{I_1}(\xi) 
+  N^{-s}\cdot \ind_{I_2}(\xi) + 
\ind_{-I_1}(\xi)  
+  N^{-s}\cdot \ind_{-I_2}(\xi) \Big\}, 
 \label{ill2}
 \end{align}
 
\noi
where $I_1$ and $I_2$ are given by 
\begin{align}
I_1 = [\al, 2\al]
\qquad \text{and}\qquad
I_2 = [N, N+\al].
\label{ill3}
\end{align}

\noi
Note that 
\begin{align}
\| \phi \|_{H^s} \sim 1.
\label{ill4}
\end{align}

% Given an interval $I \subset \R$, let $\P_I$ be the (spatial)
% frequency projector onto $\{\xi \in I\}$ defined by 
%\[\ft{\P_I f}(\xi) = \ind_I(\xi) \cdot \ft f(\xi) .\]

From \eqref{ill3}, we have $I_1 + I_2 = [N + \al, N+3\al]$.
Then, from \eqref{der2}, \eqref{ill2}, and \eqref{ill3}
with \eqref{sILW4a}, \eqref{res10}, and \eqref{res9}, we have
\begin{align}
\begin{split}
 \ind_{I_1 + I_2} \cdot \F_x \big(\dd_\eps^2 v^\high(t; \eps)|_{\eps = 0} \big)
(\xi)
 =-  \frac{2\xi e^{i t \wt p_\dl(\xi)}}{\al N^s}
 \intt_{\substack{\xi = \xi_1+ \xi_2\\
 \xi_1 \in I_1\\\xi_2 \in I_2}}
 \frac{e^{-i t \wt \Xi_\dl(\bar \xi)} - 1}{\wt \Xi_\dl(\bar \xi)}
 d\xi_1 .
 \end{split}
\label{ill5}
\end{align}

From
Lemma \ref{LEM:res}\,(ii) and \eqref{ill1}, we have 
\begin{align*}
|\wt \Xi_\dl(\bar \xi)|\sim \dl^{-1} |\xi_1 \xi_2|
\sim N^{-\ta}.
%\label{ill6}
\end{align*}

\noi
Then, by a Taylor expansion, we have
\begin{align}
-  \frac{e^{- i t \wt \Xi_\dl(\bar \xi)} - 1}{\wt \Xi_\dl(\bar \xi)}
 = it + O(t^2 N^{-\ta})
= it + o(1),  
 \label{ill7}
\end{align}

\noi
as $N \to \infty$.
Recall from 
\cite[Lemma 3.5]{OhFE}
that 
\begin{align}
\ind_{a + I}* \ind_{b + I} (\xi)\ges |I| \cdot \ind_{a+b+I}(\xi)
\label{ill8}
\end{align}
		
\noi
for any $a, b, \xi \in \R$ and 
any interval $I \subset \R$.
Finally, if we suppose that the  map: $\phi \in H^s(\R)\mapsto
v^\high (t) \in H^s(\R)$ is $C^2$, 
then it follows
 from \eqref{ill4}, \eqref{der2}, \eqref{ill5}, 
and \eqref{ill7} with \eqref{ill8} and \eqref{ill1} that
\begin{align*}
1\sim \|\phi\|_{H^s}^2 \ges
\big\|\dd_\eps^2 v^\high(t; \eps)|_{\eps = 0} \big\|_{H^s}
\ges_t \frac{N^{1+s} \al^\frac 32} {\al N^s}
= \al^{\frac 12} N
\sim \dl^{\frac 12} N^{\frac {1-\ta} 2}
\too \infty, 
\end{align*}

\noi
as $N \to \infty$ (for each fixed $0 < \dl \le 1$
and $t \in \R\setminus\{0\}$),
which is a contradiction.
This proves 
Theorem \ref{THM:1}\,(ii).

%%%%%%%%%%%%%%%%%%%%%%%%%%%%%%%%%%%%%%
%%%%%%%%%%%%%%%%%%%%%%%%%%%%%%%%%%%%%%
%%%%%%%%%%%%%%%%%%%%%%%%%%%%%%%%%%%%%%
%%%%%%%%%%%%%%%%%%%%%%%%%%%%%%%%%%%%%%
\begin{ackno}\rm

The authors would like to thank Thierry Laurens
for a comment on Remark \ref{REM:circle}.
A.\,C.~was supported by  CNRS-INSMI through a grant ``PEPS Jeunes chercheurs et jeunes chercheuses 2025''.
B.\,H.-G. was supported by NSF grant DMS-2406816.
G.\,L. was supported by the NSFC (grant no.~12501181).
T.\,O.~was supported by the European Research Council (grant no.~864138 ``SingStochDispDyn")
and
acknowledges support from  
the NSFC (grant no.~W2531005).

\end{ackno}

\end{document}